\documentclass[reqno,11pt]{amsart}
\usepackage{amsmath,amssymb,latexsym,cancel,rotating,hyperref}

\usepackage{eucal}

\usepackage{amscd}
\usepackage[dvips]{color}
\usepackage{multicol}
\usepackage[all]{xy}           %xypic macro for latex2.09
\usepackage{graphicx}
\usepackage{color}
\usepackage{colordvi}
\usepackage{xspace}
\usepackage{tikz}

\usepackage{enumitem}

\numberwithin{equation}{section}

\newtheorem{theorem}{Theorem}[section]
\newtheorem{proposition}[theorem]{Proposition}

\newtheorem{corollary}[theorem]{Corollary}
\newtheorem{lemma}[theorem]{Lemma}

\newtheorem{remark}[theorem]{Remark}

\newtheorem{definition}[theorem]{Definition}

\input xy
\xyoption{poly}
\xyoption{2cell}
\xyoption{all}

\def\ZZ{\mathbb Z}

\def\e{{\bf{e}}}

\def\b{{\bf{b}}}

\def\ZZ{\mathbb{Z}}

\renewcommand{\eqref}[1]{{\rm (\ref{#1})}}

\begin{document}

\title[On the acyclic quantum  cluster algebras with principal coefficients]
{On the acyclic quantum  cluster algebras with principal coefficients}

\author{Junyuan Huang, Xueqing Chen, Ming Ding and Fan Xu}
\address{School of Mathematics and Information Science\\
Guangzhou University, Guangzhou 510006, P.R.China}
\email{jy4545@e.gzhu.edu.cn (J.Huang)}
\address{Department of Mathematics,
 University of Wisconsin-Whitewater\\
800 W. Main Street, Whitewater, WI.53190. USA}
\email{chenx@uww.edu (X.Chen)}
\address{School of Mathematics and Information Science\\
Guangzhou University, Guangzhou 510006, P.R.China}
\email{dingming@gzhu.edu.cn (M.Ding)}
\address{Department of Mathematical Sciences\\
Tsinghua University\\
Beijing 100084, P.~R.~China} \email{fanxu@mail.tsinghua.edu.cn (F.Xu)}

%\keywords{$\mathbb{Z}[q^{\pm \frac{1}{2}}]-$basis, quantum cluster
%algebra.}

%\date{April 10, 2010}
%\thanks{Ming Ding was supported by NSF of China (No. 11301282) and Specialized Research Fund for the Doctoral Program of Higher Education (No. 20130031120004) and Fan Xu was supported by NSF of China (No. 11471177).}

%    General info
%\subjclass[2000]{Primary  16G20, 17B67; Secondary  17B35, 18E30}

%\date{\today}


\maketitle

\begin{abstract}
In this paper, we focus on a new  lower bound quantum cluster algebra which is generated by the initial quantum cluster variables and the quantum projective cluster variables of an acyclic quantum cluster algebra with principal coefficients. We show that the new lower bound quantum cluster algebra  coincides with the corresponding acyclic quantum cluster algebra.   Moreover, we establish a class of formulas between these generators, and obtain the dual PBW basis of this algebra.
\end{abstract}

%\tableofcontents

\section{Background}

Fomin and Zelevinsky \cite{ca1,ca2} invented cluster algebras in order
to set up an algebraic framework for studying the total positivity \cite{Lus3} and the dual canonical bases in coordinate rings and their $q$-deformations \cite{Lus1, Lus2}.
A cluster algebra is  a subalgebra of the field of rational functions in $n$ variables generated by a family of generators called cluster variables which are obtained recursively through mutation.
Quantum cluster algebras as the quantization of cluster algebras,
were later introduced by Berenstein and Zelevinsky \cite{AA}. One of the most important problems in cluster theory is to  construct bases of (quantum) cluster algebras with good properties (see for example \cite{Q3,li-pan}).

The lower bound cluster algebra, introduced by Berenstein, Fomin and
Zelevinsky \cite{bfz}, is a subalgebra of the associated cluster algebra generated by these cluster variables in the clusters which are one-step mutations from the initial cluster. They showed that a cluster algebra   is equal to
the lower bound  if and only if it  possesses an~acyclic seed, and then the standard monomial basis of this kind of~cluster algebra can  be naturally obtained. In the quantum setting, these results also hold  \cite{AA}. Note that the standard monomial basis is useful to construct well-behaved basis for a quantum cluster algebra which contains all cluster monomials, i.e., monomials with variables from the same cluster \cite{BZ}.

 A new lower bound cluster algebra was introduced by Baur and Nasr-Isfahani \cite{BN},  which is generated by  the initial cluster variables and the so-called projective cluster variables. Note that this new lower bound is constructed by using cluster variables from two seeds,  which is the main difference from Berenstein, Fomin and
Zelevinsky's lower bound. They proved that  an acyclic cluster algebra coincides with its  lower bound, and further constructed a basis for this cluster algebra.   Qin pointed out that this basis should agree with an associated dual
PBW basis in the language of \cite{KQ}, which is important for defining the triangular basis \cite{Q1,Q2}.

In this paper, we work with the acyclic quantum cluster algebra with principal coefficients. We  first show that  the  lower bound algebra generated by the initial quantum cluster variables and the quantum projective cluster variables  coincides with the corresponding acyclic quantum cluster algebra. Then an interesting family of identities between the initial quantum cluster variables and the quantum projective cluster variables are obtained.  As an application, we  construct  the dual PBW basis of this  algebra. It is worth remarking that our approach in this paper is almost exclusively  elementary algebraic computation.

\section{Preliminaries}

First of all, we fix our terminology and briefly recall some definitions and related results of the quantum cluster algebras.

For two positive integers $s$ and $t$ with $s<t$, for convention  we write $[s,t]$ for the set $\{s,s+1,\ldots,t-1,t\}$.
A square integer  matrix $B$ is called skew-symmetrizable if there exists
some integer diagonal matrix $D$ with positive diagonal entries such that $DB$ is skew-symmetric and we then call $D$ the skew-symmetrizer of $B$.
Let $m$ and $n$ be two positive integers with $m\geq n$. Let $\widetilde{B}=(b_{ij})$ be an~$m\times n$ integer matrix with its upper $n\times n$ submatrix being skew-symmetrizable denoted by $B$ called the principal part  of $\widetilde{B}$.
 We can choose  an $m\times m$  skew-symmetric integer matrix $\Lambda $ such that $\widetilde{B}^{T}\Lambda=\begin{bmatrix} D & {\bf{0}} \end{bmatrix}$ for some integer diagonal
 matrix  $D$ with positive diagonal entries. The pair $(\widetilde{B},\Lambda)$ is called a compatible pair, which is specified by the following data:
\begin{enumerate}
\item an $m\times n$ integer matrix $\widetilde{B}$ with the skew-symmetrizable principal part $B$ and its skew-symmetrizer $D=\operatorname{diag}(d_1, d_2, \dots, d_n)$ and
$\widetilde{B}=\begin{bmatrix}  {\bf{b}}_1 & {\bf{b}}_2 &\dots &{\bf{b}}_n \end{bmatrix}$ with
${\bf{b}}_j \in \mathbb{Z}^m $ for $ j \in[1, n]$;
\item a skew-symmetric bilinear form $\Lambda:\mathbb{Z}^m\times\mathbb{Z}^m\to\mathbb{Z}$ satisfying the compatibility
condition with  $\widetilde{B}$, i.e.,
\begin{equation}\label{(dj)}
	\Lambda({\bf{b}}_j, {\bf{e}}_i)=\delta_{ij}d_j \quad (i \in [1,m], j\in [1,n])
		\end{equation}
where $\e_i$ is the $i$-th unit vector in $\mathbb{Z}^m$ for any $ i\in [1,m]$.
\end{enumerate}

Note that we can identify the bilinear form $\Lambda$ with the skew-symmetric $m\times m$ matrix still denoted by $\Lambda=(\lambda_{ij})$ with
$\lambda_{ij}:=\Lambda(\e_i,\e_j).$

Let $q$ be a formal variable and denote $\mathbb{Z}[q^{ \pm {\frac{1}{2}}}]$ the ring of integer Laurent polynomials
in the variable $q^{\frac{1}{2}}$. The based quantum torus $\mathcal{T}=\mathcal{T}(\Lambda)$ is the $\mathbb{Z}[q^{ \pm {\frac{1}{2}}}]$-algebra with a
distinguished  $\mathbb{Z}[q^{ \pm {\frac{1}{2}}}]$-basis $\{ {X^{\e}}: {\e}\in\mathbb{Z}^m \}$ and the multiplication given by
\begin{equation}\label{(ef)}
	X^{\e} X^{\bf{f}}=q^{\frac{\Lambda({\bf{e}}, {\bf{f}})}{2}} X^{{\bf{e}} +{\bf{f}}} \quad (\e, {\bf{f}} \in \mathbb{Z}^m).
	 \end{equation}
%for $\e, f \in\mathbb{Z}^m$.

Denote by $x_i=X^{\e_i}$ for any ${i\in [1,m]}$, then
the elements of $x_i$ and their inverses generate $\mathcal{T}$ as a  $\mathbb{Z}[q^{ \pm {\frac{1}{2}}}]$-algebra, subject to the quasi-commutative relations
\begin{equation}\label{(ij)}
	{x_i}{x_j} = {q^{{\lambda _{ij}}}}{x_j}{x_i}
\end{equation}
for $i,j \in [1,m]$.
For any ${\bf{a}}=(a_1, a_2, \dots, a_m)\in \mathbb{Z}^m$,  we get $$X^{\bf{a}}=q^{\frac{1}{2} \sum_{l<k} a_k a_l \lambda_{kl}} x_1^{a_1}x_2^{a_2}\cdots x_m^{a_m}.$$

\begin{definition}\label{def of generalized seed}
With the above notations, a quantum seed is defined to be the triple $\big(\widetilde{\mathbf{x}},\Lambda,\widetilde{B}\big)$, where
the set $\widetilde{\mathbf{x}}=\{x_{1},x_2,\ldots, x_{m}\}$ is the extended cluster, $\mathbf{x}=\{x_{1},x_2,\ldots,x_{n}\}$ is the cluster,
elements $x_i$ for $i\in [1,n]$ are called quantum cluster variables and elements $x_{i}$ for $i\in [n+1, m]$ are called frozen variables.
\end{definition}

Define the function
\begin{gather*}
[x]_+:=
 \begin{cases}
 x,  &\text{if}\quad x\geq 0;
 \\
0,  & \text{if}\quad x< 0.
 \end{cases}
\end{gather*}

\begin{definition}
For $k \in [1,n]$, the mutation of~a quantum seed $\big(\widetilde{\mathbf{x}},\Lambda,\widetilde{B}\big)$ in the direction $k$ is the quantum seed \smash{$\mu_k\big(\widetilde{\mathbf{x}},\Lambda,\widetilde{B}\big):=\big(\widetilde{\mathbf{x}}',
\Lambda',\widetilde{B}'\big)$}, where
\begin{enumerate}\itemsep=0pt
\item the set $\widetilde{\mathbf{x}}':=(\widetilde{\mathbf{x}}- \{x_k\})\cup\{x'_{k}\}$ with
\begin{equation}\label{(er)}
x'_k=X^{-\e_k+[\b_k]_+}+X^{-\e_k+[-\b_k]_+};
\end{equation}
\item
the matrix $\widetilde{B}':=\mu_k\big(\widetilde{B}\big)$ is defined by
\begin{gather}\label{(B)}
b'_{ij}=
 \begin{cases}
 -b_{ij}, &\text{ if }  i=k \text{ or } j=k;
 \\
b_{ij}+\displaystyle\frac{|b_{ik}|b_{kj}+b_{ik}|b_{kj}|}{2}, &\text{otherwise};
 \end{cases}
\end{gather}
\item
the skew-symmetric matrix $\Lambda':=\mu_k\big(\Lambda\big)$ is defined by
\begin{gather}\label{(A)}
	\lambda'_{ij}=
	\begin{cases}
		\lambda_{ij}, &\text{ if } i,j\ne k;
		\\
		- {\lambda _{ij}} + \sum\limits_{t = 1}^m {{{[{b_{ti}}]}_ + }{\lambda _{tj}}} , &\text{ if } i= k, j\ne k.
	\end{cases}
\end{gather}
\end{enumerate}
\end{definition}

Note that $\mu_k$ is an involution.  Two quantum seeds are called mutation-equivalent if one can be obtained from
another by a sequence of mutations. Denote the skew-field of fractions of $\mathcal{T}$  by $\mathcal{F}$ and
$$\mathop{\mathbb{ZP}}:=\mathbb{Z}[q^{ \pm {\frac{1}{2}}}][x^{\pm}_{n+1},\ldots,x^{\pm}_m].$$

\sloppy\begin{definition}\label{def of gca}
Given an initial quantum seed $\big(\widetilde{\mathbf{x}}, \Lambda, \widetilde{B}\big)$, the quantum cluster algebra $\smash{\mathcal{A}\big(\widetilde{\mathbf{x}}, \Lambda, \widetilde{B}\big)}$ is the $\mathop{\mathbb{ZP}}$-subalgebra of~$\mathcal{F}$ generated by all quantum cluster variables from all quantum seeds mutation-equivalent to $\smash{\big(\widetilde{\mathbf{x}},\Lambda,\widetilde{B}\big)}$.
\end{definition}
Note that  one can recover the classical cluster algebra by setting $q=1$.

The directed graph associated to a quantum seed~$\big(\widetilde{\mathbf{x}}, \Lambda,\widetilde{B}\big)$ is denoted by   $\Gamma\big(\widetilde{\mathbf{x}}, \Lambda,\widetilde{B}\big)$ with vertices $[1,n]$ and the directed edges from $i$ to $j$ if $b_{ij}>0$.

\begin{definition}
 A   quantum seed $\big(\widetilde{\mathbf{x}},\Lambda,\widetilde{B}\big)$ is called acyclic if $\Gamma\big(\widetilde{\mathbf{x}},\Lambda,\widetilde{B}\big)$ does not contain any oriented cycle.
A quantum cluster algebra is called acyclic if it has an acyclic  quantum seed.
\end{definition}

\begin{definition}
	A standard monomial in $x_1,x_1^{\prime},\ldots,x_n,x_n^{\prime}$ is an element of the
	form $x_1^{a_1}\cdots x_n^{a_n}(x_1^{\prime})^{a_1^{\prime}}\cdots(x_n^{\prime})^{a_n^{\prime}}$, where all exponents are non-negative integers, and	$a_ka_k^{\prime}=\overset{.}{\operatorname*{0}}$ for $k \in [1, n]$.
\end{definition}

Denote by $\mathcal{L}(\widetilde{\mathbf{x}},\Lambda,\widetilde{B}) := \ZZ\mathbb{P}[{x_1},x_1^{\prime}, \ldots {x_n},x_n^{\prime}]$.
The following theorems are quantum versions of the corresponding results in \cite{bfz}.

\begin{theorem}\label{standard-1}\cite[Theorem 7.3]{AA}
	The standard monomials in $x_{1},x_{1}^{\prime},\ldots,x_{n},x_{n}^{\prime}$ are linearly independent
	over $\mathbb{ZP}$  (i.e., they form a $\mathbb{ZP}$-basis of $\mathcal{L}(\widetilde{\mathbf{x}},\Lambda,\widetilde{B})$) if and only if the quantum seed $\big(\widetilde{\mathbf{x}},\Lambda,\widetilde{B}\big)$ is acyclic.
\end{theorem}

\begin{theorem}\label{LA}\cite[Theorem 7.6]{AA}
	The condition that a quantum seed $\big(\widetilde{\mathbf{x}},\Lambda,\widetilde{B}\big)$ is acyclic, is necessary and
	sufficient for the equality $\mathcal{L}(\widetilde{\mathbf{x}},\Lambda,\widetilde{B}) = \smash{\mathcal{A}\big(\widetilde{\mathbf{x}}, \Lambda, \widetilde{B}\big)}$.
\end{theorem}

\section{The quantum projective cluster variables}
Up to simultaneously reordering of columns and rows, we can assume  that the entries in the skew-symmetrizable matrix $B$ satisfy  $b_{ij} \geq 0$ for any $i>j$ which then defines a linear order  $\triangleleft$ on $[1,n]$.
In the following, let $\Sigma=( {\widetilde{\mathbf{x}},\Lambda,{\widetilde B}})$ be an acyclic  quantum seed of an acyclic  quantum cluster algebra with principal coefficients.

\begin{definition}\label{3.1}
For any $i \in [1,n],$  define  a new  acyclic quantum seed
$${\Sigma ^{(i)}} = ({\widetilde{\mathbf{x}}^{(i)}},{\Lambda ^{(i)}},{\widetilde B^{(i)}}): = {\mu _{{i}}} \cdots {\mu _{{2}}}{\mu _{{1}}}(\widetilde{\mathbf{x}},\Lambda ,\widetilde B),$$
where ${\widetilde{\mathbf{x}}^{(i)}} =\{ {x_1^{(i)},x_2^{(i)}, \ldots ,x_n^{(i)},x_{n + 1}^{(i)} \ldots ,x_{2n}^{(i)}}\}$.
Thus we have ${\Sigma ^{(n)}} = ({\widetilde{\mathbf{x}}^{(n)}},{\Lambda ^{(n)}},{\widetilde B^{(n)}})$.
The  cluster ${\mathbf{x}}^{(n)} =\{ {x_1^{(n)},x_2^{(n)}, \ldots ,x_n^{(n)}}\}$ is called  the quantum projective cluster, and each  cluster variable in $\mathbf{x}^{(n)}$ is called a quantum projective cluster variable.
\end{definition}

Note that the new  quantum seed ${\Sigma^{(i)}}$ is obtained by applying a sequence of mutations on $\Sigma$ corresponding to a sink sequence of the directed graph $\Gamma(\Sigma)$ and
 by the above definition we have ${B^{(n)}}={B}$, $x_i^{(i)} = x_i^{(j)}$ for any $1\leq i\leq j\leq n$ and $x_i^{(j)}=x_i$ for any $i\in [n+1,2n]$ and $j \in [1,n]$.

It is straightforward to obtain that \[{\widetilde{B}^{(i)}} =\scriptsize\left[ {\begin{array}{*{20}{c}}
		0&{{b_{12}}}& \cdots &{{b_{1i}}}&{ - {b_{1\, i + 1}}}&{ - {b_{1\,i + 2}}}& \cdots &{ - {b_{1\,n - 1}}}&{ - {b_{1n}}}\\
		{{b_{21}}}&0& \cdots &{{b_{2i}}}&{ - {b_{2\,i + 1}}}&{ - {b_{2\,i + 2}}}& \cdots &{ - {b_{2\,n - 1}}}&{ - {b_{2n}}}\\
		\vdots & \vdots & \ddots & \vdots & \vdots & \vdots & \ddots & \vdots & \vdots \\
		{{b_{i1}}}&{{b_{i2}}}& \cdots &0&{ - {b_{i\,i + 1}}}&{ - {b_{i\,i + 2}}}& \cdots &{ - {b_{i\,n - 1}}}&{ - {b_{in}}}\\
		{ - {b_{i + 1\,1}}}&{ - {b_{i + 1\,2}}}& \cdots &{ - {b_{i + 1\,i}}}&0&{{b_{i + 1\,i + 2}}}& \cdots &{{b_{i + 1\,n - 1}}}&{{b_{i + 1\,n}}}\\
		{ - {b_{i + 2\,1}}}&{ - {b_{i + 2\,2}}}& \cdots &{ - {b_{i + 2\,i}}}&{{b_{i + 2\,i + 1}}}&0& \cdots &{{b_{i + 2\,n - 1}}}&{{b_{i + 2\,n}}}\\
		\vdots & \vdots & \ddots & \vdots & \vdots & \vdots & \ddots & \vdots & \vdots \\
		{ - {b_{n - 1\,1}}}&{ - {b_{n - 1\,2}}}& \cdots &{ - {b_{n - 1\,i}}}&{{b_{n - 1\,i + 1}}}&{{b_{n - 1\,i + 2}}}& \cdots &0&{{b_{n - 1\,n}}}\\
		{ - {b_{n1}}}&{ - {b_{n2}}}& \cdots &{ - {b_{ni}}}&{{b_{n\,i + 1}}}&{{b_{n\,i + 2}}}& \cdots &{{b_{n\,n-1}}}&0\\
		{ - 1}&{}&{}&{}&{}&{}&{}&{}&{} \\
		{}&{ - 1}&{}&{}&{}&{}&{}&{}&{}\\
		{}&{}& \ddots &{}&{}&{}&{}&{}&{}\\
		{}&{}&{}&{ - 1}&{}&{}&{}&{}&{}\\
		{}&{}&{}&{}&1&{}&{}&{}&{}\\
		{}&{}&{}&{}&{}&1&{}&{}&{}\\
		{}&{}&{}&{}&{}&{}& \ddots &{}&{}\\
		{}&{}&{}&{}&{}&{}&{}&1&{}\\
		{}&{}&{}&{}&{}&{}&{}&{}&1
\end{array}} \right],\]

\[{\Lambda ^{(i)}} =\scriptsize\left[ {\begin{array}{*{20}{c}}
		0&{{\lambda _{12}}}& \cdots &{{\lambda _{1i}}}&{ - {\lambda _{1\,i + 1}}}&{ - {\lambda _{1\,i + 2}}}& \cdots &{ - {\lambda _{1\,2n - 1}}}&{ - {\lambda _{1\,2n}}}\\
		{{\lambda _{21}}}&0& \cdots &{{\lambda _{2i}}}&{ - {\lambda _{2\,i + 1}}}&{-{\lambda _{2\,i + 2}}}& \cdots &{ - {\lambda _{2\,2n - 1}}}&{ - {\lambda _{2\,2n}}}\\
		\vdots & \vdots & \ddots & \vdots & \vdots & \vdots & \ddots & \vdots & \vdots \\
		{{\lambda _{i1}}}&{{\lambda _{i2}}}& \cdots &0&{ - {\lambda _{i\,i + 1}}}&{ - {\lambda _{i\,i + 2}}}& \cdots &{ - {\lambda _{i\,2n - 1}}}&{ - {\lambda _{i\,2n}}}\\
		{ - {\lambda _{i + 1\,1}}}&{ - {\lambda _{i + 1\,2}}}& \cdots &{ - {\lambda _{i + 1\,i}}}&0&{{\lambda _{i + 1\,i + 2}}}& \cdots &{{\lambda _{i + 1\,2n - 1}}}&{{\lambda _{i + 1\,2n}}}\\
		{ - {\lambda _{i + 2\,1}}}&{ - {\lambda _{i + 2\,2}}}& \cdots &{ - {\lambda _{i + 2\,i}}}&{{\lambda _{i + 2\,i + 1}}}&0& \cdots &{{\lambda _{i + 2\,2n - 1}}}&{{\lambda _{i + 2\,2n}}}\\
		\vdots & \vdots & \ddots & \vdots & \vdots & \vdots & \ddots & \vdots & \vdots \\
		{ - {\lambda _{2n - 1\,1}}}&{ - {\lambda _{2n - 1\,2}}}& \cdots &{ - {\lambda _{2n - 1\,i}}}&{{\lambda _{2n - 1\,i + 1}}}&{{\lambda _{2n - 1\,i + 2}}}& \cdots &0&{{\lambda _{2n - 1\,2n}}}\\
		{ - {\lambda _{2n\,1}}}&{ - {\lambda _{2n\,2}}}& \cdots &{ - {\lambda _{2n\,i}}}&{{\lambda _{2n\,i + 1}}}&{{\lambda _{2n\,i + 2}}}& \cdots &{{\lambda _{2n\,2n - 1}}}&0
\end{array}} \right].\]

\begin{definition}\label{3.2}
Given an acyclic  quantum seed $ \Sigma$. The lower bound quantum cluster algebra ${\mathcal{L}^{(n)}}( \Sigma)$
is defined to be the algebra generated by all initial and quantum projective cluster variables  over $\ZZ\mathbb{P}$, i.e.,
$${\mathcal{L}^{(n)}}(\Sigma) = \ZZ\mathbb{P}[{x_1},x_1^{(n)}, \dots, {x_n}, x_n^{(n)}].$$
\end{definition}

It is obvious that ${\mathcal{L}^{(n)}}(\Sigma)$ is a finitely generated subalgebra of the quantum cluster algebra $\mathcal{A}(\Sigma)$.

The following well-known result will be used later.
\begin{lemma}\label{lemmaxy}
	If $xy=q^{-1}yx$, then ${(x + y)^n} = \sum\limits_{k = 0}^n {{{\left[ {\begin{array}{*{20}{c}}
						n\\
						k
				\end{array}} \right]}_q}} {x^k}{y^{n - k}},$
where	
$ {\left[ {\begin{array}{*{20}{c}}
				n\\
			k
		\end{array}} \right]} _q = \frac{{{{[n]}_q}{{[n - 1]}_q} \cdots {{[n - k + 1]}_q}}}{{{{[k]}_q}{{[k - 1]}_q} \cdots {{[1]}_q}}}$
and $[n]_q=\frac{q^n-1}{q-1}.$
\end{lemma}

The notation $ \prod\limits_{j \in [1,m]}^ \triangleleft$ means that the product is taken in increasing order with
respect to $\triangleleft$. For any ${\bf{a}}=(a_1, a_2, \dots, a_{2n})\in \mathbb{Z}^{2n}$ and $i \in [1,n],$  define $$({X^{(i)}})^{\bf{a}}:=q^{\frac{1}{2} \sum_{l<k} a_k a_l \lambda^{(i)}_{kl}}\prod\limits_{j \in [1,2n]}^ \triangleleft (x^{(i)}_j)^{a_j}.$$

\begin{lemma}\label{Lemma.r}
For any $k \in [1,n],$ 	we have
\begin{eqnarray}	\label{xin1}
		{({X^{(k - 1)}})^{ - {\e_k} + {{[\b_k^{(k - 1)}]}_ + }}}	&	=& 	{q^{\frac{1}{2}\big( { - \sum\limits_{t = 1}^{k - 1} {\sum\limits_{j = t + 1}^{k - 1} {{b_{tk}}{b_{jk}}\lambda _{tj}^{}} }  + \sum\limits_{j = 1}^{k - 1} {{b_{jk}}\lambda _{jk}^{}}  + \sum\limits_{j = k + 1}^n {{b_{jk}}\lambda _{kj}^{}}  + \lambda _{k\,n + k}^{}} \big)}}\nonumber\\
		&&	\cdot   {\prod\limits_{j \in [1,k - 1]}^ \triangleleft  {{{(x_j^{(k - 1)})}^{ - {b_{jk}}}}} }\cdot x_k^{ - 1}{({X^{(k - 1)}})^{\sum\limits_{j = k + 1}^n {{b_{jk}}{\e_j}}  + {\e_{n + k}}}}.
	\end{eqnarray}
\end{lemma}

\begin{proof}
 A direct computation shows that
\begin{eqnarray*}
	&&	{({X^{(k - 1)}})^{ - {\e_k} + {{[\b_k^{(k - 1)}]}_ + }}}= {({X^{(k - 1)}})^{( - \sum\limits_{j = 1}^{k - 1} {{b_{jk}}{\e_j}} ) - {\e_k} + (\sum\limits_{j = k + 1}^n {{b_{jk}}{\e_j}} ) + {\e_{n + k}}}}\\
	&	=& {q^{ - \frac{1}{2}\big( {{\Lambda ^{(k - 1)}}( - \sum\limits_{j = 1}^{k - 1} {{b_{jk}}{\e_j}} , - \sum\limits_{t = j + 1}^{k - 1} {{b_{tk}}{\e_t}}  - {\e_k}) + {\Lambda ^{(k - 1)}}( - \sum\limits_{j = 1}^{k - 1} {{b_{jk}}{\e_j}} ,\sum\limits_{t = k + 1}^n {{b_{tk}}{\e_t}} + {\e_{n + k}})} \big)}}\\
	&&	\cdot {q^{ - \frac{1}{2}\big( {{\Lambda ^{(k - 1)}}( - {\e_k},\sum\limits_{j = k + 1}^n {{b_{jk}}{\e_j}}  + \e_{n + k})} \big)}}\cdot
	\prod\limits_{j \in [1,k - 1]}^ \triangleleft  {{{(x_j^{(k - 1)})}^{ - {b_{jk}}}}}\cdot x_k^{ - 1}{({X^{(k - 1)}})^{\sum\limits_{j = k + 1}^n {{b_{jk}}{\e_j}} + {\e_{n + k}}}}.
	\end{eqnarray*}
Note that we have
	\begin{eqnarray*}\label{(11)}
		&&{\Lambda ^{(k - 1)}}( - \sum_{j = 1}^{k - 1} {{b_{jk}}{\e_j}} , - \sum_{t =j+ 1}^{k - 1} {{b_{tk}}{\e_t}}  - {\e_k})\\
		&=& \sum_{j = 1}^{k - 1} {\sum\limits_{t = j + 1}^{k - 1} {{b_{jk}}{b_{tk}}\lambda _{jt}^{(k - 1)}} }  + \sum\limits_{j = 1}^{k - 1} {{b_{jk}}\lambda _{jk}^{(k - 1)}} = \sum\limits_{t = 1}^{k - 1} {\sum\limits_{j = t + 1}^{k - 1} {{b_{tk}}{b_{jk}}\lambda _{tj}^{}} }  - \sum\limits_{j = 1}^{k - 1} {{b_{jk}}\lambda _{jk}^{}},
	\end{eqnarray*}
	\begin{eqnarray*}\label{(12)}
		&&{\Lambda ^{(k - 1)}}( - \sum\limits_{j = 1}^{k - 1} {{b_{jk}}{\e_j}} ,\sum\limits_{t = k + 1}^n {{b_{tk}}{\e_t}}  + {\e_{n + k}})\\
		&=& \sum\limits_{j = 1}^{k - 1} {\sum\limits_{t = k + 1}^n {{b_{jk}}{b_{tk}}\lambda _{jt}^{}} }  + \sum\limits_{j = 1}^{k - 1} {{b_{jk}}\lambda _{j\,n + k}^{}}= \sum\limits_{j = 1}^{k - 1} {{b_{jk}}(\sum\limits_{t= k + 1}^n {{b_{tk}}\lambda _{jt}^{}}  + \lambda _{j\,n + k}^{})}\\
		&= &\sum\limits_{j = 1}^{k - 1} {{b_{jk}}\Lambda ({\e_j},  \sum\limits_{t = k + 1}^n {{b_{tk}}{\e_t}} +{\e_{n + k}})} \mathop  = \limits^{(\ref{(dj)})} -\sum\limits_{j = 1}^{k - 1} {{b_{jk}}\Lambda ({\e_j},\sum\limits_{t= 1}^{k - 1} {{b_{tk}}{\e_t}} )} \\
		&=&- \sum\limits_{j = 1}^{k - 1} {\sum\limits_{t= 1}^{k - 1} {{b_{jk}}{b_{tk}}{\lambda _{jt}}}  = 0},
	\end{eqnarray*}
and	
	\begin{eqnarray*}\label{(13)}
		{\Lambda ^{(k - 1)}}( - {\e_k},\sum\limits_{j = k + 1}^n {{b_{jk}}{\e_j}}  + {\e_{n + k}})&	= & - \sum\limits_{j = k + 1}^n {{b_{jk}}\lambda _{kj}^{}} - \lambda _{k\,n + k}^{}.
	\end{eqnarray*}
 Hence, the proof follows immediately.\end{proof}

\begin{proposition}\label{prop}
The algebra $\mathbb{ZP}[{x_1},x'_1, \dots, {x_n}, x'_n]$ is a $\mathbb{ZP}$-subalgebra of ${\mathcal{L}^{(n)}}(\Sigma).$
\end{proposition}
\begin{proof}
For any $i \in [1,n],$ according to the formula (\ref{(er)}), we have	
\begin{eqnarray}\label{(x')}
	{x'_i} &=& {X^{ - {\e_i} + {{[{\b_i}]}_ + }}} + {X^{ - {\e_i} + {{[ - {\b_i}]}_ + }}}\nonumber\\
	&= &{X^{ - {\e_i} + \sum\limits_{j = i + 1}^n {{b_{ji}}{\e_j}}  + {\e_{n + i}}}} + {X^{ - {\e_i} - \sum\limits_{j = 1}^{i - 1} {{b_{ji}}{\e_j}} }}\nonumber\\
	&=& {q^{\frac{1}{2}(  {\lambda _{i\, n + i}} + \sum\limits_{j = i + 1}^n {{b_{ji}}{\lambda _{ij}}} )}}x_i^{ - 1}{X^{\sum\limits_{j = i + 1}^n {{b_{ji}}{\e_j}}  + {\e_{n + i}}}}\nonumber\\
	 &&+ {q^{\frac{1}{2}( - \sum\limits_{t = 1}^{i - 1} {\sum\limits_{j = t + 1}^{i - 1} {{b_{ti}}{b_{ji}}{\lambda _{tj}}} }  - \sum\limits_{t = 1}^{i - 1} {{b_{ti}}{\lambda _{ti}}} )}}\prod\limits_{j \in [1,i - 1]}^ \triangleleft  {x_j^{ - {b_{ji}}}}\cdot  x_i^{ - 1}.
\end{eqnarray}

Now we consider ${\Sigma^{(i)}} =( {{\widetilde{\mathbf{x}} ^{(i)}},{\Lambda ^{(i)}},{{\widetilde B}^{(i)}}} )={\mu _{i-1}} ( {{\widetilde{\mathbf{x}}^{(i-1)}},{\Lambda ^{(i-1)}},{{\widetilde B}^{(i-1)}}} )$.
According to the formula (\ref{xin1}), we have
\begin{eqnarray}\label{(xkk1)}
	x_i^{(i)} &=& {q^{\frac{1}{2}\big( { - \sum\limits_{t = 1}^{i - 1} {\sum\limits_{r = t + 1}^{i - 1} {{b_{ti}}{b_{ri}}\lambda _{tr}^{}} }  + \sum\limits_{t = 1}^{i - 1} {{b_{ti}}\lambda _{ti}^{}}  + \sum\limits_{r = i + 1}^n {{b_{ri}}\lambda _{ir}^{}}  + \lambda _{i\,n + i}^{}} \big)}}\nonumber\\
	&&\cdot  {\prod\limits_{r \in [1,i - 1]}^ \triangleleft  {{{(x_r^{(i - 1)})}^{ - {b_{ri}}}}} }\cdot x_i^{ - 1}{({X^{(i - 1)}})^{\sum\limits_{r = i + 1}^n {{b_{ri}}{\e_r}}  + {\e_{n + i}}}} + x_i^{ - 1}.
\end{eqnarray}
We have  the following two cases.
\begin{enumerate}
\item When $b_{ri}=0$ for all $r \in [1,i-1]$, we obtain
\begin{eqnarray*}
	x_i^{(i)} &=& {q^{\frac{1}{2}\big( {\sum\limits_{r = i + 1}^n {{b_{ri}}\lambda _{ir}^{}}  + \lambda _{i\,n + i}^{}} \big)}}x_i^{ - 1}{({X^{(i - 1)}})^{\sum\limits_{r = i + 1}^n {{b_{ri}}{\e_r}} + {\e_{n + i}}}} + x_i^{ - 1}
	\mathop  = \limits^{(\ref{(x')})} {x'_i}.
\end{eqnarray*}
Thus ${x'_i} \in \mathbb{ZP}[{x_1},x_1^{(n)}, \ldots {x_n},x_n^{(n)}];$
\item When $b_{ri}\ne0$  for some $r \in [1,i - 1]$. According to  the formula (\ref{(xkk1)}), we  have
\begin{eqnarray*}
		x_i^{(i)} &=& {q^{\frac{1}{2}\big( { - \sum\limits_{t = 1}^{i - 1} {\sum\limits_{r = t + 1}^{i - 1} {{b_{ti}}{b_{ri}}\lambda _{tr}^{}} }  + \sum\limits_{t = 1}^{i - 1} {{b_{ti}}\lambda _{ti}^{}}  + \sum\limits_{r = i + 1}^n {{b_{ri}}\lambda _{ir}^{}}  + \lambda _{i\,n + i}^{}} \big)}}\\
		&&\cdot  {\prod_{r \in [1,i - 1]}^ \triangleleft  {{{({A_r} + {x_r^{ - 1}} )}^{ - {b_{ri}}}}} } \cdot x_i^{ - 1}{({X^{(i - 1)}})^{\sum\limits_{r = i + 1}^n {{b_{ri}}{\e_r}}  + {\e_{n + i}}}} + x_i^{ - 1},
\end{eqnarray*}
where ${A_r} = {q^{{\# _r}}}\cdot \prod\limits_{j \in [1,r - 1]}^ \triangleleft  {{{(x_j^{(r - 1)})}^{ - {b_{jr}}}}} \cdot x_r^{ - 1}\cdot \prod\limits_{j \in [r + 1,n]}^ \triangleleft  {x_j^{{b_{jr}}}} \cdot x_{n + r}^{}$ and ${\# _r} \in \frac{\mathbb{Z}}{2}.$

Multiplying both sides of the above equation from the left by $\prod\limits_{r \in [1,i - 1]}^ \triangleleft  {x_r^{ - {b_{ri}}}} $, we have
\begin{eqnarray*}
&&	\prod\limits_{r \in [1,i - 1]}^ \triangleleft  {x_r^{ - {b_{ri}}}} \cdot x_i^{(i)}\\
& =& {q^{\frac{1}{2}\big( { - \sum\limits_{t = 1}^{i - 1} {\sum\limits_{r = t + 1}^{i - 1} {{b_{ti}}{b_{ri}}\lambda _{tr}^{}} }  + \sum\limits_{t = 1}^{i - 1} {{b_{ti}}\lambda _{ti}^{}}  + \sum\limits_{r = i + 1}^n {{b_{ri}}\lambda _{ir}^{}}  + \lambda _{i\,n + i}^{}} \big)}} \cdot \prod\limits_{r \in [1,i - 1]}^ \triangleleft  {x_r^{ - {b_{ri}}}} \\
&&\cdot {\prod\limits_{r \in [1,i - 1]}^ \triangleleft  {{{({A_r} + x_r^{ - 1})}^{ - {b_{ri}}}}} } \cdot x_i^{ - 1}{({X^{(i - 1)}})^{\sum\limits_{r = i + 1}^n {{b_{ri}}{\e_r}}  + {\e_{n + i}}}}+ \prod\limits_{r \in [1,i - 1]}^ \triangleleft  {x_r^{ - {b_{ri}}}}\cdot x_i^{ - 1}.
\end{eqnarray*}

 According to Lemma \ref{lemmaxy}, we obtain 	
\begin{eqnarray*}	
&&	\prod\limits_{r \in [1,i - 1]}^ \triangleleft  {x_r^{ - {b_{ri}}}} \cdot x_i^{(i)}\\
	&	=& {q^{\frac{1}{2}\big( { - \sum\limits_{t = 1}^{i - 1} {\sum\limits_{r = t + 1}^{i - 1} {{b_{ti}}{b_{ri}}\lambda _{tr}^{}} }  + \sum\limits_{t = 1}^{i - 1} {{b_{ti}}\lambda _{ti}^{}}  + \sum\limits_{r = i + 1}^n {{b_{ri}}\lambda _{ir}^{}}  + \lambda _{i\,n + i}^{}} \big)}}	\cdot  \prod\limits_{r \in [1,i - 1]}^ \triangleleft  {x_r^{ - {b_{ri}}}}\\
	&&\cdot  \prod\limits_{r \in [1,i - 1]}^ \triangleleft  {x_r^{{b_{ri}}}} \cdot
	x_i^{ - 1}{({X^{(i - 1)}})^{\sum\limits_{r = i + 1}^n {{b_{ri}}{\e_r}}  + {\e_{n + i}}}}
+ \prod\limits_{r \in [1,i - 1]}^ \triangleleft  {x_r^{ - {b_{ri}}}} \cdot x_i^{ - 1} + f,
\end{eqnarray*}
where
\begin{eqnarray*}
	f &= & {q^*} \cdot \prod\limits_{r \in [1,i - 1]}^ \triangleleft  {x_r^{ - {b_{ri}}}}\cdot       \prod\limits_{r \in [1,i - 1]}^ \triangleleft  \big(\sum\limits_{h = 1}^{ - {b_{ri}}} {{C_{rh}} \cdot \prod\limits_{j \in [1,r - 1]}^ \triangleleft  {{{(x_j^{(r - 1)})}^{ - h{b_{jr}}}} \cdot x_r^{{b_{ri}}}\cdot \prod\limits_{j \in [r + 1,n]}^ \triangleleft  {x_j^{h{b_{jr}}}} \cdot x_{n + r}^h} } \big) \\
	&&	\cdot x_i^{ - 1}{({X^{(i - 1)}})^{\sum\limits_{r = i + 1}^n {{b_{ri}}{\e_r}}  + {\e_{n + i}}}},
\end{eqnarray*}
here ${*} \in \frac{\mathbb{Z}}{2}$, ${C_{rh}} \in \mathbb{Z}[{q^{ \pm \frac{1}{2}}}]$ and for convention denote $\sum\limits_{h = 1}^{0}:=1$.

Note that $\prod\limits_{r \in [1,i - 1]}^ \triangleleft  {x_r^{ - {b_{ri}}}} \cdot \prod\limits_{r \in [1,i - 1]}^ \triangleleft  {x_r^{  {b_{ri}}}} =
{q^{\sum\limits_{t = 1}^{i - 1} {\sum\limits_{r = t + 1}^{i - 1} {  {b_{ti}}{b_{ri}}{\lambda _{tr}}}}}}$.
Thus, we have
\begin{eqnarray*}
&&	\prod\limits_{r \in [1,i - 1]}^ \triangleleft  {x_r^{ - {b_{ri}}}} \cdot x_i^{(i)}\\
 &= &{q^{\frac{1}{2}\big( { - \sum\limits_{t = 1}^{i - 1} {\sum\limits_{r = t + 1}^{i - 1} {{b_{ti}}{b_{ri}}\lambda _{tr}^{}} }  + \sum\limits_{t = 1}^{i - 1} {{b_{ti}}\lambda _{ti}^{}}  + \sum\limits_{r = i + 1}^n {{b_{ri}}\lambda _{ir}^{}}  + \lambda _{i\,n + i}^{}} \big)}}\\
	&&\cdot {q^{\sum\limits_{t = 1}^{i - 1} {\sum\limits_{r = t + 1}^{i - 1} {{b_{ti}}{b_{ri}}{\lambda _{tr}}} } }}x_i^{ - 1}{({X^{(i - 1)}})^{\sum\limits_{r = i + 1}^n {{b_{ri}}{\e_r}}  + {\e_{n + i}}}} + \prod\limits_{r \in [1,i - 1]}^ \triangleleft  {x_r^{ - {b_{ri}}}} \cdot x_i^{ - 1} + f\\
&	=& {q^{\frac{1}{2}\big( {\sum\limits_{t = 1}^{i - 1} {\sum\limits_{r = t + 1}^{i - 1} {{b_{ti}}{b_{ri}}\lambda _{tr}^{}} }  + \sum\limits_{t = 1}^{i - 1} {{b_{ti}}\lambda _{ti}^{}}  + \sum\limits_{r = i + 1}^n {{b_{ri}}\lambda _{ir}^{}}  + \lambda _{i\,n + i}^{}} \big)}}\\
&&	\cdot x_i^{ - 1}{({X^{(i - 1)}})^{\sum\limits_{r = i + 1}^n {{b_{ri}}{\e_r}}  + {\e_{n + i}}}} + \prod\limits_{r \in [1,i - 1]}^ \triangleleft  {x_r^{ - {b_{ri}}}} \cdot x_i^{ - 1} + f\\
& = &{q^{\frac{1}{2}\big(\sum\limits_{t = 1}^{i - 1} {\sum\limits_{r = t + 1}^{i - 1} {{b_{ti}}{b_{ri}}\lambda _{tr}^{}} }  + \sum\limits_{t = 1}^{i - 1} {{b_{ti}}\lambda _{ti}^{}} \big)}}\big({q^{\frac{1}{2}\big(\sum\limits_{r = i + 1}^n {{b_{ri}}\lambda _{ir}^{}} + \lambda _{i\,n + i}^{}\big)}}x_i^{ - 1}{X^{\sum\limits_{r = i + 1}^n {{b_{ri}}{\e_r}}  + {\e_{n + i}}}}\\
&&	+ {q^{\frac{1}{2}\big( - \sum\limits_{t = 1}^{i - 1} {\sum\limits_{r = t + 1}^{i - 1} {{b_{ti}}{b_{ri}}{\lambda _{tr}}} }  - \sum\limits_{t = 1}^{i - 1} {{b_{ti}}{\lambda _{ti}}} \big)}}\prod\limits_{r \in [1,i - 1]}^ \triangleleft  {x_r^{ - {b_{ri}}}} \cdot x_i^{ - 1}\big) + f\\
&	\mathop  = \limits^{(\ref{(x')})} &{q^{\frac{1}{2}\big(\sum\limits_{t = 1}^{i - 1} {\sum\limits_{r = t + 1}^{i - 1} {{b_{ti}}{b_{ri}}\lambda _{tr}^{}} }  + \sum\limits_{t = 1}^{i - 1} {{b_{ti}}\lambda _{ti}^{}} \big)}}{x'_i} + f.
\end{eqnarray*}

We  claim that $f \in \mathbb{ZP}\left[ {{x_1}, x_1^{(1)}, \dots ,x_{i - 1},x_{i - 1}^{(i - 1)}},x_{i +1},\dots,{x_n} \right].$

We rewrite $f$ as
\begin{eqnarray*}
	f&\mathop  = \limits^{(\ref{(ij)})}&{q^*}\prod\limits_{r \in [1,i - 1]}^ \triangleleft  {x_r^{ - {b_{ri}}}} \cdot \prod\limits_{r \in [1,i - 1]}^ \triangleleft  \big({x_r^{{b_{ri}}}\big(\sum\limits_{h = 1}^{ - {b_{ri}}} {{G_{rh}}\prod\limits_{j \in [1,r - 1]}^ \triangleleft  {{{(x_j^{(r - 1)})}^{ - h{b_{jr}}}}} } \cdot \prod\limits_{j \in [r + 1,n]}^ \triangleleft  {x_j^{h{b_{jr}}}} \cdot x_{n + r}^h}\big) \big)\\
	&&\cdot x_i^{ - 1}{({X^{(i - 1)}})^{\sum\limits_{r = i + 1}^n {{b_{ri}}{\e_r}}  + {\e_{n + i}}}}\\
	&\mathop  = \limits^{(\ref{(ij)})}& {q^*}\prod\limits_{r \in [1,i - 1]}^ \triangleleft  {x_r^{ - {b_{ri}}}} \cdot \prod\limits_{r \in [1,i - 1]}^ \triangleleft  {x_r^{{b_{ri}}}} \cdot \prod\limits_{r \in [1,i - 1]}^ \triangleleft  {\big(\sum\limits_{h = 1}^{ - {b_{ri}}} {{P_{rh}}\prod\limits_{j \in [1,r - 1]}^ \triangleleft  {{{(x_j^{(r - 1)})}^{ - h{b_{jr}}}}} } } \\
&&\cdot \prod\limits_{j \in [r + 1,n]}^ \triangleleft  {x_j^{h{b_{jr}}}} \cdot x_{n + r}^h \big) \cdot x_i^{ - 1}{({X^{(i - 1)}})^{\sum\limits_{r = i + 1}^n {{b_{ri}}{\e_r}}  + {\e_{n + i}}}}\\
&	\mathop  = & {q^{{\bullet}}}\prod\limits_{r \in [1,i - 1]}^ \triangleleft  {\big(\sum\limits_{h = 1}^{ - {b_{ri}}} {{P_{rh}}\prod\limits_{j \in [1,r - 1]}^ \triangleleft  {{{(x_j^{(r - 1)})}^{ - h{b_{jr}}}}} } } \cdot \prod\limits_{j \in [r + 1,n]}^ \triangleleft  {x_j^{h{b_{jr}}}} \cdot x_{n + r}^h\big)\\
	&&\cdot x_i^{ - 1}{({X^{(i - 1)}})^{\sum\limits_{r = i + 1}^n {{b_{ri}}{\e_r}}  + {\e_{n + i}}}},
\end{eqnarray*}
where  ${\bullet} \in \frac{\mathbb{Z}}{2}$ and $G_{rh}, P_{rh} \in \mathbb{Z}[{q^{ \pm \frac{1}{2}}}].$

Note that there exists  some $r \in [1,i - 1]$ such that $b_{ri}\ne0$. Without loss of generality, assume that  $b_{1i}\ne 0$,  then we have
\begin{eqnarray*}
		f	&\mathop  = \limits^{(\ref{(ij)})}& {q^{{\bullet}}}\big(\sum\limits_{h = 1}^{ - {b_{1i}}} {{S_{1h}}\prod\limits_{j \in [2,n],j \ne i}^ \triangleleft  {x_j^{h{b_{j1}}}} \cdot x_{n + 1}^h}\cdot x_i^{h{b_{i1}} - 1}\big)\cdot x_i^{}\\
		&&\cdot \prod\limits_{r \in [2,i - 1]}^ \triangleleft  {\big(\sum\limits_{h = 1}^{ - {b_{ri}}} {{P_{rh}}\prod\limits_{j \in [1,r - 1]}^ \triangleleft  {{{(x_j^{(r - 1)})}^{ - h{b_{jr}}}}} } } \cdot \prod\limits_{j \in [r + 1,n]}^ \triangleleft  {x_j^{h{b_{jr}}}} \cdot x_{n + r}^h\big)\\
		&&\cdot x_i^{ - 1}{({X^{(i - 1)}})^{\sum\limits_{r = i + 1}^n {{b_{ri}}{\e_r}}  + {\e_{n + i}}}}\\
		&\mathop  = \limits^{(\ref{(ij)})}& {q^{{\bullet}}}\big(\sum\limits_{h = 1}^{ - {b_{1i}}} {{S_{1h}}\prod\limits_{j \in [2,n],j \ne i}^ \triangleleft  {x_j^{h{b_{j1}}}} \cdot x_{n + 1}^h}\cdot x_i^{h{b_{i1}} - 1}\big)\\
		&&\cdot \prod\limits_{r \in [2,i - 1]}^ \triangleleft  {\big(\sum\limits_{h = 1}^{ - {b_{ri}}} {{S_{rh}}\prod\limits_{j \in [1,r - 1]}^ \triangleleft  {{{(x_j^{(r - 1)})}^{ - h{b_{jr}}}}} } }\cdot  \prod\limits_{j \in [r + 1,n]}^ \triangleleft  {x_j^{h{b_{jr}}}} \cdot x_{n + r}^h\big)\\
		&&\cdot {({X^{(i - 1)}})^{\sum\limits_{r = i + 1}^n {{b_{ri}}{\e_r}}  + {\e_{n + i}}}},
\end{eqnarray*}
where  $ S_{rh}\in \mathbb{Z}[{q^{ \pm \frac{1}{2}}}].$
Then the claim follows,
and thus $x'_i \in \mathbb{ZP}[{x_1},x_1^{(n)}, \ldots ,{x_n},x_n^{(n)}].$
\end{enumerate}
The proof is completed.\end{proof}

\begin{theorem}\label{thm1}
If ${\Sigma}=( {\widetilde{\mathbf{x}},\Lambda,{\widetilde B}})$ is an acyclic  quantum seed, then $$\mathcal{A}( {\Sigma})={\mathcal{L}^{(n)}}( \Sigma).$$
\end{theorem}

\begin{proof}
It follows from Theorem \ref{LA} and Proposition \ref{prop}.
\end{proof}

\section{The dual PBW bases}
In this section, we first establish
a class of formulas for acyclic  quantum cluster algebras with principal coefficients, and then
  construct the dual PBW bases of these algebras.

\begin{lemma}\label{519}
For any $k\in [1,n]$ and $l \in \mathbb{N}$, we have
	\begin{equation}\label{(xknk1)}
	{(x_k^{(n)})^l}{x_k}  - {x_k}{(x_k^{(n)})^l} =	f_{k,l}^{(k)},
	\end{equation}
where
$f_{k,l}^{(k)} = {g_{k,k,l}}\prod\limits_{j \in [k + 1,2n]}^ \triangleleft  {x_j^{{b_{jk}}}} \cdot \prod\limits_{j \in [1,k - 1]}^ \triangleleft  {{{(x_j^{(n)})}^{ - {b_{jk}}}}}  \cdot  {(x_k^{(n)})^{l - 1}}$
and ${g_{k,k,l}}\in \mathbb{Z}[{q^{ \pm \frac{1}{2}}}].$	

\end{lemma}

\begin{proof}
According to the formula (\ref{(er)}), we have
	\begin{eqnarray*}
		x_k^{(n)}{x_k}  &=& {({X^{(k - 1)}})^{ - {\e_k} + {{[\b_k^{(k - 1)}]}_ + }}}{x_k} + {({X^{(k - 1)}})^{ - {\e_k} + {{[ - \b_k^{(k - 1)}]}_ + }}}{x_k}\\
		&=& {({X^{(k - 1)}})^{ - {\e_k} + \b_k^{(k - 1)}}}{x_k} + {({X^{(k - 1)}})^{ - {\e_k}}}{x_k}
	%&	= &{({\widetilde{X}^{(k - 1)}})^{ - {e_k} + b_k^{(k - 1)}}}{({\widetilde{X}^{(k - 1)}})^{{e_k}}} + x_k^{ - 1}{x_k}\\
	%&=& {q^{\frac{1}{2}\Lambda^{(k-1)} ( - {e_k} + b_k^{(k - 1)},{e_k})}}{({\widetilde{X}^{(k - 1)}})^{b_k^{(k - 1)}}} + 1\\
	%&	=& {q^{\frac{1}{2}\Lambda^{(k-1)} (b_k^{(k - 1)},{e_k})}}{({\widetilde{X}^{(k - 1)}})^{b_k^{(k - 1)}}} + 1\\
		={q^{\frac{d_k}{2}}}{({X^{(k - 1)}})^{\b_k^{(k - 1)}}} + 1.
\end{eqnarray*}

Similarly,	we have
	${x_k}x_k^{(n)}	= {q^{ - \frac{d_k}{2}}}{({X^{(k - 1)}})^{\b_k^{(k - 1)}}} + 1.$
Thus
	\begin{eqnarray*}
	x_k^{(n)}{x_k} - {x_k}x_k^{(n)} &=& {q^{ - \frac{d_k}{2}}}({q^{{d_k}}} - 1){({X^{(k - 1)}})^{\b_k^{(k - 1)}}}\\
&\mathop  = \limits^{(\ref{(ef)})} &{g_{k,k,1}} \prod\limits_{j \in [k + 1,2n]}^ \triangleleft  {x_j^{{b_{jk}}}}\cdot\prod\limits_{j \in [1,k - 1]}^ \triangleleft  {{{(x_j^{(n)})}^{ - {b_{jk}}}}} =:f_{k,1}^{(k)},	
\end{eqnarray*}
where $g_{k,k,1} \in \mathbb{Z}[{q^{ \pm \frac{1}{2}}}].$

For $l\geq 2$, we have
\begin{eqnarray*}
		{(x_k^{(n)})^l}{x_k}&  =  & {x_k}{(x_k^{(n)})^l} + 	f_{k,l}^{(k)},
\end{eqnarray*}
where
$f_{k,l}^{(k)} = \sum\limits_{r = 1}^l {{{(x_k^{(n)})}^{l - r}}f_{k,1}^{(k)}{{(x_k^{(n)})}^{r - 1}}}.$

Then, we have		
	\begin{eqnarray*}		
	f_{k,l}^{(k)} 	&=& \sum\limits_{r = 1}^l {{g_{k,k,1}} {{(x_k^{(n)})}^{l - r}}\cdot\prod\limits_{j \in [k + 1,2n]}^ \triangleleft  {x_j^{{b_{jk}}}}\cdot\prod\limits_{j \in [1,k - 1]}^ \triangleleft  {{{(x_j^{(n)})}^{ - {b_{jk}}}}} \cdot {{(x_k^{(n)})}^{r - 1}}} \\
		&	\mathop  = \limits^{(\ref{(ij)})}& \sum\limits_{r = 1}^l {{q^{ - (l - r){d_k}}}{g_{k,k,1}}\prod\limits_{j \in [k + 1,2n]}^ \triangleleft  {x_j^{{b_{jk}}}} \cdot\prod\limits_{j \in [1,k - 1]}^ \triangleleft  {{{(x_j^{(n)})}^{ - {b_{jk}}}}} \cdot  {{(x_k^{(n)})}^{l - 1}}} \\
			&= &{g_{k,k,l}} \prod\limits_{j \in [k + 1,2n]}^ \triangleleft  {x_j^{{b_{jk}}}}\cdot \prod\limits_{j \in [1,k - 1]}^ \triangleleft  {{{(x_j^{(n)})}^{ - {b_{jk}}}}}  \cdot {(x_k^{(n)})^{l - 1}},
\end{eqnarray*}
where ${g_{k,k,l}}   \in \mathbb{Z}[{q^{ \pm \frac{1}{2}}}].$	

The proof is completed.\end{proof}

\begin{lemma}\label{520}
For any $j \in[1,n]$ and $l\in \mathbb{N}$, we have
		\begin{equation}\label{(xjnj-1)}
			{(x_j^{(n)})^l}{x_{j - 1}} - {q^{ - l{\lambda _{j\, j - 1}}}}{x_{j - 1}}{(x_j^{(n)})^l} = f_{j - 1,l}^{(j)},
		\end{equation}
		where $f_{j - 1,1}^{(j)} \in \mathbb{ZP}[x_1^{(n)}, \dots ,x_{j-1}^{(n)},{x_{j }}, \dots, {x_n}],$ and $f_{j - 1,l}^{(j)} \in \mathbb{ZP}[x_1^{(n)}, \dots ,x_{j-1}^{(n)},x_{j}^{(n)},{x_{j }}, \dots, {x_n}]$ for any $l\geq 2$.
\end{lemma}			
\begin{proof}
We have
	\begin{eqnarray*}
	x_j^{(n)}{x_{j - 1}}	&\mathop  = \limits^{(\ref{xin1})} & {q^{\frac{1}{2}\big( { - \sum\limits_{k = 1}^{j - 1} {\sum\limits_{t = k + 1}^{j - 1} {{b_{kj}}{b_{tj}}\lambda _{kt}^{}} }  + \sum\limits_{k = 1}^{j - 1} {{b_{kj}}\lambda _{kj}^{}}  + (\sum\limits_{k = j + 1}^n {{b_{kj}}\lambda _{jk}^{}} ) + \lambda _{j\, n + j}^{}} \big)}}\\
		&&	\cdot \prod\limits_{k \in [1,j - 1]}^ \triangleleft  {{{(x_k^{(n)})}^{ - b_{kj}^{}}}} \cdot x_j^{ - 1}{({X^{(j - 1)}})^{\sum\limits_{t = j + 1}^n {{b_{tj}}{\e_t}}  + {\e_{n + j}}}}{x_{j - 1}} + x_j^{ - 1}{x_{j - 1}}\\
		&\mathop  = \limits^{(\ref{(ij)})}&{q^{{M_j} - {\lambda _{j\,j - 1}} + \sum\limits_{t = j + 1}^n {b_{tj}^{}{\lambda _{t\, j - 1}}}  + {\lambda _{n + j\,j - 1}}}}\prod\limits_{k \in [1,j - 1]}^ \triangleleft  {{{(x_k^{(n)})}^{ - b_{kj}^{}}}} \cdot {x_{j - 1}}\\
		&&	\cdot x_j^{ - 1}{({X^{(j - 1)}})^{\sum\limits_{t = j + 1}^n {{b_{tj}}{\e_t}} + {\e_{n + j}}}} + {q^{ - {\lambda _{j\,j - 1}}}}{x_{j - 1}}x_j^{ - 1},
	\end{eqnarray*}	
where	${M_j} = \frac{1}{2}\big( { - \sum\limits_{k = 1}^{j - 1} {\sum\limits_{t = k + 1}^{j - 1} {{b_{kj}}{b_{tj}}\lambda _{kt}^{}} }  + \sum\limits_{k = 1}^{j - 1} {{b_{kj}}\lambda _{kj}^{}}  + \sum\limits_{k = j + 1}^n {{b_{kj}}\lambda _{jk}^{}}  + \lambda _{j\,n + j}^{}} \big).$

Note that
	\begin{eqnarray}\label{(123456)}
		\sum\limits_{t = j + 1}^n {b_{tj}^{}{\lambda _{t\,j - 1}}}  + {\lambda _{n + j\,j - 1}} + \sum\limits_{k = 1}^{j - 2} {b_{kj}^{}{\lambda _{k\,j - 1}}} & = &\Lambda ({\b_j},{\e_{j - 1}}) = 0.
	\end{eqnarray}
Thus, we  have
	\begin{eqnarray*}
	x_j^{(n)}{x_{j - 1}}	&=& {q^{{M_j} - {\lambda _{j\, j - 1}}{\rm{ - }}\sum\limits_{k = 1}^{j - 2} {b_{kj}^{}{\lambda _{k\, j - 1}}} }}\prod\limits_{k \in [1,j - 2]}^ \triangleleft  {{{(x_k^{(n)})}^{ - b_{kj}^{}}}}\cdot {(x_{j - 1}^{(n)})^{ - b_{j - 1\,j}^{}}}{x_{j - 1}}\\
		&&\cdot x_j^{ - 1}{({X^{(j - 1)}})^{\sum\limits_{t = j + 1}^n {{b_{tj}}{\e_t}}  + {\e_{n + j}}}} + {q^{ - {\lambda _{jj - 1}}}}{x_{j - 1}}x_j^{ - 1}\\
		&	\mathop  = \limits^{(\ref{(xknk1)})}& {q^{{M_j} - {\lambda _{j\,j - 1}}{\rm{ - }}\sum\limits_{k = 1}^{j - 2} {b_{kj}^{}{\lambda _{k\,j - 1}}} }}\prod\limits_{k \in [1,j - 2]}^ \triangleleft  {{{(x_k^{(n)})}^{ - b_{kj}^{}}}} \\
		&&\cdot \left( {{x_{j - 1}}{{(x_{j - 1}^{(n)})}^{ - b_{j - 1\,j}^{}}} + f_{j - 1, - {b_{j - 1\,j}}}^{(j - 1)}} \right)x_j^{ - 1}{({X^{(j - 1)}})^{\sum\limits_{t = j + 1}^n {{b_{tj}}{\e_t}}  + {\e_{n + j}}}}\\
		&&	+ {q^{ - {\lambda _{j\,j - 1}}}}{x_{j - 1}}x_j^{ - 1}\\
		&= &{q^{{M_j} - {\lambda _{j\,j - 1}}{\rm{ - }}\sum\limits_{k = 1}^{j - 2} {b_{kj}^{}{\lambda _{k\,j - 1}}} }}\prod\limits_{k \in [1,j - 2]}^ \triangleleft  {{{(x_k^{(n)})}^{ - b_{kj}^{}}}} \cdot {x_{j - 1}}{(x_{j - 1}^{(n)})^{ - b_{j - 1\,j}^{}}}\\
		&&\cdot x_j^{ - 1}{({X^{(j - 1)}})^{(\sum\limits_{t = j + 1}^n {{b_{tj}}{\e_t}} ) + {\e_{n + j}}}} + {q^{ - {\lambda _{j\,j - 1}}}}{x_{j - 1}}x_j^{ - 1} + f_{j - 1,1}^{(j)}\\
		&\mathop  = \limits^{(\ref{(ij)})}& {q^{{M_j} - {\lambda _{j\,j - 1}}}}  {x_{j - 1}}\prod\limits_{k \in [1,j - 1]}^ \triangleleft  {{{(x_k^{(n)})}^{ - b_{kj}^{}}}} \cdot x_j^{ - 1}{({X^{(j - 1)}})^{\sum\limits_{t = j + 1}^n {{b_{tj}}{\e_t}}  + {\e_{n + j}}}} \\
		&&+ {q^{ - {\lambda _{j\,j - 1}}}}{x_{j - 1}}x_j^{ - 1} +f_{j - 1,1}^{(j)},	
	\end{eqnarray*}
where
	\begin{eqnarray}\label{ff}
	f_{j - 1,1}^{(j)}	&=& {q^{{M_j} - {\lambda _{j\,j - 1}}{\rm{ - }}\sum\limits_{k = 1}^{j - 2} {b_{kj}^{}{\lambda _{k\,j - 1}}} }}\nonumber\\
		&&\cdot \prod\limits_{k \in [1,j - 2]}^ \triangleleft  {{{(x_k^{(n)})}^{ - b_{kj}^{}}}}\cdot f_{j - 1, - {b_{j - 1\,j}}}^{(j - 1)}x_j^{ - 1}{({X^{(j - 1)}})^{\sum\limits_{t = j + 1}^n {{b_{tj}}{\e_t}}  + {\e_{n + j}}}}.
	\end{eqnarray}

We claim that  $	f_{j - 1,1}^{(j)} \in \mathbb{ZP}[x_1^{(n)}, \dots, x_{j - 1}^{(n)},{x_j}, \dots, {x_n}].$

(i) When $b_{j-1\,j}=0$, it is easy to see that $f_{j - 1,1 }^{(j)}=0;$

(ii) When $b_{j-1\,j}\neq 0$, by Lemma \ref{519}, we have
	\begin{eqnarray*}
	f_{j - 1,1}^{(j)}	&= &{g_1}\prod\limits_{k \in [1,j - 2]}^ \triangleleft  {{{(x_k^{(n)})}^{ - b_{kj}^{}}}} \cdot \prod\limits_{k \in [j,2n]}^ \triangleleft  {x_k^{{b_{k\,j - 1}}}} \cdot \prod\limits_{k \in [1,j - 2]}^ \triangleleft  {{{(x_k^{(n)})}^{ - {b_{k\,j - 1}}}}}\\
		&&\cdot {(x_{j - 1}^{(n)})^{- {b_{j - 1\,j}} - 1}}x_j^{ - 1}{({X^{(j - 1)}})^{\sum\limits_{t = j + 1}^n {{b_{tj}}{\e_t}} + {\e_{n + j}}}}\\
		&\mathop  = \limits^{(\ref{(ij)})}& {g_{j ,j - 1,1}}x_j^{{b_{j\,j - 1}} - 1}\prod\limits_{k \in [j + 1,2n]}^ \triangleleft  {x_k^{{b_{k\,j - 1}} + {b_{kj}}}}\cdot \prod\limits_{k \in [1,j - 2]}^ \triangleleft  {{{(x_k^{(n)})}^{ - b_{kj}^{} - {b_{k\,j - 1}}}}}\\
		&&\cdot {(x_{j - 1}^{(n)})^{- {b_{j - 1\,j}} - 1}},
	\end{eqnarray*}
where $g_1$ and ${g_{j ,j - 1,1}} \in \mathbb{Z}[{q^{ \pm \frac{1}{2}}}].$ Thus the claim follows.
	
	Note that
	\begin{eqnarray*}
	& &	{q^{ - {\lambda _{j\,j - 1}}}}{x_{j - 1}}x_j^{(n)}\\
	&  = & {q^{{M_j} - {\lambda _{j\,j - 1}}}}{x_{j - 1}}\prod\limits_{k \in [1,j - 1]}^ \triangleleft  {{{(x_k^{(n)})}^{ - b_{kj}^{}}}}\cdot x_j^{ - 1}%\nonumber\\
	 {({X^{(j - 1)}})^{\sum\limits_{t = j + 1}^n {{b_{tj}}{\e_t}}  + {\e_{n + j}}}} + {q^{ - {\lambda _{j\,j - 1}}}}{x_{j - 1}}x_j^{ - 1}.
	\end{eqnarray*}
	
Therefore,  we have $x_j^{(n)}{x_{j - 1}} - {q^{ - {\lambda _{j\,j - 1}}}}{x_{j - 1}}x_j^{(n)} = f_{j - 1,1}^{(j)}.$
	
For $l\geq 2$,  we have
	\begin{eqnarray*}
		{(x_j^{(n)})^l}{x_{j - 1}}&\mathop  - & {q^{ - l{\lambda _{j\,j - 1}}}}{x_{j - 1}}{(x_j^{(n)})^l} =f_{j - 1,l}^{(j)},
	\end{eqnarray*}
where $f_{j - 1,l}^{(j)} = \sum\limits_{t = 1}^l {{q^{ - (t - 1){\lambda _{j\,j - 1}}}}{{(x_j^{(n)})}^{l - t}}f_{j - 1,1}^{(j)}} {(x_j^{(n)})^{t - 1}}.$

Thus, the proof is completed. \end{proof}

\begin{lemma}\label{5201}
 For any $j \in [1,n]$ and $ k< j,$ we have
\begin{equation}\label{(js+1)}
	{x_j^{(n)}}{x_{k}} - {q^{ - {\lambda _{jk}}}}{x_{k}}{x_j^{(n)}} =f_{k,1}^{(j)},
\end{equation}
where $f_{k,1}^{(j)}{\rm{ }} = M_{k,1}^{(j)}{({X^{(j - 1)}})^{\sum\limits_{v = j + 1}^n {{b_{vj}}{\e_v}}  + {\e_{n + j}}}}$ and

\begin{eqnarray*}
	M_{k,1}^{(j)}
	&= &\sum\limits_{t = k}^{j - 1} {\big({q^{{M_j} - {\lambda _{jk}} + {\lambda _{n + j\,k}} + \sum\limits_{v = j + 1}^n {{b_{vj}}{\lambda _{vk}}}  + \sum\limits_{r = t + 1}^{j - 1} {b_{rj}^{}{\lambda _{rk}}} }}\prod\limits_{i \in [1,t - 1]}^ \triangleleft  {{{(x_i^{(n)})}^{ - b_{ij}^{}}}} \cdot f_{k, - {b_{tj}}}^{(t)}} \\
	&&\cdot \prod\limits_{i \in [t + 1,j - 1]}^ \triangleleft  {{{(x_i^{(n)})}^{ - b_{ij}^{}}}} \cdot x_j^{ - 1}\big)  \in \mathbb{ZP}[x_1^{(n)}, \dots ,x_{j - 1}^{(n)},x_{k + 1}, \dots ,{x_n}],
\end{eqnarray*}
with $f_{k,- {b_{tj}}}^{(t)} := \sum\limits_{i = 1}^{- {b_{tj}}} {{q^{ - (i - 1){\lambda _{tk}}}}{{(x_t^{(n)})}^{- {b_{tj}} - i}}f_{k,1}^{(t)}} {(x_t^{(n)})^{i - 1}}.$			
\end{lemma}

\begin{proof}
We prove it  by induction. When $k=j-1,$ according to Lemma \ref{520}, we have
$${x_j^{(n)}}{x_{j - 1}} - {q^{ - {\lambda _{j\,j - 1}}}}{x_{j - 1}}{x_j^{(n)}} =f_{j - 1,1}^{(j)}.$$
Note that, following the proof of Lemma \ref{520}, we can rewrite $f_{j - 1,1}^{(j)}$ as
$$f_{j - 1,1}^{(j)} = M_{j - 1,1}^{(j)}{({X^{(j - 1)}})^{\sum\limits_{v = j + 1}^n {{b_{vj}}{\e_v}}  + {\e_{n + j}}}},$$ where
\begin{eqnarray*}
	M_{j - 1,1}^{(j)} &=& {q^{{M_j} - {\lambda _{j\,j - 1}}{\rm{ - }}\sum\limits_{v = 1}^{j - 2} {b_{vj}^{}{\lambda _{v\,j - 1}}} }}\prod\limits_{i \in [1,j - 2]}^ \triangleleft  {{{(x_i^{(n)})}^{ - b_{ij}^{}}}} \cdot f_{j - 1, - {b_{j - 1\,j}}}^{(j - 1)}x_j^{ - 1} \\
&	\mathop  = \limits^{(\ref{(123456)})}&{q^{{M_j} - {\lambda _{j\,j - 1}} + {\lambda _{n + j\,j - 1}} + \sum\limits_{v = j + 1}^n {{b_{vj}}{\lambda _{v\,j - 1}}} }}\prod\limits_{i \in [1,j - 2]}^ \triangleleft  {{{(x_i^{(n)})}^{ - b_{ij}^{}}}} \cdot f_{j - 1, - {b_{j - 1\,j}}}^{(j - 1)}x_j^{ - 1}
\end{eqnarray*}
which is in $\mathbb{ZP}[x_1^{(n)}, \dots, x_{j - 1}^{(n)},{x_j}, \dots,{x_n}].$

Suppose that for any $j \in [1,n]$ and $ k\in [s+1,j-1],$ the statement (\ref{(js+1)}) is true.
Thus we have
\begin{eqnarray}\label{(jls+1)}
		{(x_j^{(n)})^l}{x_k}&  = & {q^{ - l{\lambda _{jk}}}}{x_k}{(x_j^{(n)})^l} + f_{k,l}^{(j)},
	\end{eqnarray}	
where
\begin{equation}\label{(qqq)}
	f_{k,l}^{(j)} = \sum\limits_{t = 1}^l {{q^{ - (t - 1){\lambda _{jk}}}}{{(x_j^{(n)})}^{l - t}}f_{k,1}^{(j)}} {(x_j^{(n)})^{t - 1}} .
\end{equation}	
	
When  $k=s$, we have
\begin{eqnarray*}
	x_j^{(n)}{x_s}&=&{q^{{M_j}}}\prod\limits_{i \in [1,j - 2]}^ \triangleleft  {{{(x_i^{(n)})}^{ - b_{ij}^{}}}} \cdot  {(x_{j - 1}^{(n)})^{ - b_{j - 1\,j}^{}}}x_j^{ - 1}{({X^{(j - 1)}})^{\sum\limits_{t = j + 1}^n {{b_{tj}}{\e_t}}  + {\e_{n + j}}}}{x_s} + x_j^{ - 1}{x_s}\\
	&	\mathop  = \limits^{(\ref{(ij)})}& {q^{{M_j} - {\lambda _{js}} + {\lambda _{n + j\, s}} + \sum\limits_{v = j + 1}^n {{b_{vj}}{\lambda _{vs}}} }}\prod\limits_{i\in [1,j - 2]}^ \triangleleft  {{{(x_i^{(n)})}^{ - b_{ij}^{}}}} \cdot {(x_{j - 1}^{(n)})^{ - b_{j - 1\,j}^{}}}{x_s}\\
	&&\cdot x_j^{ - 1}{({X^{(j - 1)}})^{\sum\limits_{t = j + 1}^n {{b_{tj}}{\e_t}}  + {\e_{n + j}}}} + {q^{ - {\lambda _{js}}}}{x_s}x_j^{ - 1}\\	
	&	\mathop  = \limits^{(\ref{(jls+1)})} &{q^{{M_j} - {\lambda _{js}} + {\lambda _{n + j\,s}} + \sum\limits_{v = j + 1}^n {{b_{vj}}{\lambda _{vs}}} }}\prod\limits_{i\in [1,j - 2]}^ \triangleleft  {{{(x_i^{(n)})}^{ - b_{ij}^{}}}} \\
	&&\cdot \big( {{q^{b_{j - 1\,j}^{}{\lambda _{j - 1\,s}}}}{x_s}{{(x_{j - 1}^{(n)})}^{ - b_{j - 1\,j}^{}}} + f_{s, - {b_{j - 1\,j}}}^{(j - 1)}} \big)x_j^{ - 1}{({X^{(j - 1)}})^{\sum\limits_{t = j + 1}^n {{b_{tj}}{\e_t}}  + {\e_{n + j}}}}\\
	&&+ {q^{ - {\lambda _{js}}}}{x_s}x_j^{ - 1}\\
	&=& {q^{{M_j} - {\lambda _{js}} + {\lambda _{n + j\,s}} + \sum\limits_{v = j + 1}^n {{b_{vj}}{\lambda _{vs}}}  + b_{j - 1\,j}^{}{\lambda _{j - 1\,s}}}}\prod\limits_{i \in [1,j - 2]}^ \triangleleft  {{{(x_i^{(n)})}^{ - b_{ij}^{}}}} \cdot {x_s}{(x_{j - 1}^{(n)})^{ - b_{j - 1\,j}^{}}}\\
	&&\cdot x_j^{ - 1}{({X^{(j - 1)}})^{\sum\limits_{t = j + 1}^n {{b_{tj}}{\e_t}}  + {\e_{n + j}}}} + {q^{ - {\lambda _{js}}}}\cdot {x_s}x_j^{ - 1} + F_{j - 1,s}^{(j)}\\
\end{eqnarray*}

\begin{eqnarray*}
	&\mathop  = \limits^{(\ref{(jls+1)})}& {q^{{M_j} - {\lambda _{js}} + {\lambda _{n + j\,s}} + \sum\limits_{v = j + 1}^n {{b_{vj}}{\lambda _{vs}}}  + \sum\limits_{t = s + 1}^{j - 1} {b_{tj}^{}{\lambda _{ts}}} }}\prod\limits_{i \in [1,s - 1]}^ \triangleleft  {{{(x_i^{(n)})}^{ - b_{ij}^{}}}} \cdot {x_s} \cdot \prod\limits_{i \in [s,j - 1]}^ \triangleleft  {{{(x_i^{(n)})}^{ - b_{ij}^{}}}} \\
	&&\cdot x_j^{ - 1}{({X^{(j - 1)}})^{\sum\limits_{t = j + 1}^n {{b_{tj}}{\e_t}}  + {\e_{n + j}}}} + {q^{ - {\lambda _{js}}}}{x_s}x_j^{ - 1} + \sum\limits_{t = s}^{j - 1} {F_{t,s}^{(j)}} \\
	&	\mathop  = \limits^{(\ref{(ij)})}& {q^{{M_j} - {\lambda _{js}} + {\lambda _{n + j\,s}} + \sum\limits_{v = j + 1}^n {{b_{vj}}{\lambda _{vs}}}  + \sum\limits_{t = 1}^{j - 1} {b_{tj}^{}{\lambda _{ts}}} }}{x_s}\prod\limits_{i\in [1,j - 1]}^ \triangleleft  {{{(x_i^{(n)})}^{ - b_{ij}^{}}}} \\
	&&\cdot x_j^{ - 1}{({X^{(j - 1)}})^{\sum\limits_{t = j + 1}^n {{b_{tj}}{\e_t}}  + {\e_{n + j}}}} + {q^{ - {\lambda _{js}}}}{x_s}x_j^{ - 1} +f_{s,1}^{(j)} \\
	&	\mathop  = \limits^{(\ref{(dj)})}&{q^{{M_j} - {\lambda _{js}}}}{x_s}\prod\limits_{i \in [1,j - 1]}^ \triangleleft  {{{(x_i^{(n)})}^{ - b_{ij}^{}}}} \cdot x_j^{ - 1}{({X^{(j - 1)}})^{\sum\limits_{t = j + 1}^n {{b_{tj}}{\e_t}}  + {\e_{n + j}}}} + {q^{ - {\lambda _{js}}}}{x_s}x_j^{ - 1} + f_{s,1}^{(j)},
\end{eqnarray*}
where
\begin{eqnarray*}
	f_{s,1}^{(j)} = \sum\limits_{t = s}^{j - 1} {F_{t,s}^{(j)}} &=&\sum\limits_{t = s}^{j - 1} {\big({q^{{M_j} - {\lambda _{js}} + {\lambda _{n + j\,s}} + \sum\limits_{v = j + 1}^n {{b_{vj}}{\lambda _{vs}}}  + \sum\limits_{r = t + 1}^{j - 1} {b_{rj}^{}{\lambda _{rs}}} }}\prod\limits_{i \in [1,t - 1]}^ \triangleleft  {{{(x_i^{(n)})}^{ - b_{ij}^{}}}} \cdot f_{s, - {b_{tj}}}^{(t)}} \nonumber\\
	&&	\cdot \prod\limits_{i\in [t + 1,j - 1]}^ \triangleleft  {{{(x_i^{(n)})}^{ - b_{ij}^{}}}} \cdot x_j^{ - 1}{({X^{(j - 1)}})^{\sum\limits_{v = j + 1}^n {{b_{vj}}{\e_v}}  + {\e_{n + j}}}}\big)\nonumber\\
	&=&M_{s,1}^{(j)}{({X^{(j - 1)}})^{\sum\limits_{v = j + 1}^n {{b_{vj}}{\e_v}}  + {\e_{n + j}}}},
\end{eqnarray*}	
with
\begin{eqnarray}\label{(888)}
	M_{s,1}^{(j)} &:=& \sum\limits_{t = s}^{j - 1} {\big({q^{{M_j} - {\lambda _{js}} + {\lambda _{n + j\,s}} + \sum\limits_{v = j + 1}^n {{b_{vj}}{\lambda _{vs}}}  + \sum\limits_{r = t + 1}^{j - 1} {b_{rj}^{}{\lambda _{rs}}} }}\prod\limits_{i \in [1,t - 1]}^ \triangleleft  {{{(x_i^{(n)})}^{ - b_{ij}^{}}}} \cdot f_{s, - {b_{tj}}}^{(t)}} \nonumber\\
	&&\cdot \prod\limits_{i \in [t + 1,j - 1]}^ \triangleleft  {{{(x_i^{(n)})}^{ - b_{ij}^{}}}}\cdot x_j^{ - 1}\big).
\end{eqnarray}

We   need to prove that  $	M_{s,1}^{(j)}  \in \mathbb{ZP}[x_1^{(n)},  \dots,x_{j - 1}^{(n)},x_{s + 1}, \dots ,{x_n}].$

Note that, if  $b_{tj}=0$ for some $t \in [s,j-1],$ then $f_{s, - {b_{tj}}}^{(t)}=0$. Thus, if  $b_{tj}=0$ for all $t \in [s,j-1],$ then $M_{s,1}^{(j)}=0.$
Therefore we only need to consider those  $t \in [s,j-1]$ such that $b_{tj}\ne 0.$   By  induction, we have
$$f_{s,1}^{(t)} = M_{s,1}^{(t)}{({X^{(t - 1)}})^{\sum\limits_{v = t + 1}^n {{b_{vt}}{\e_v}}  + {\e_{n + t}}}} = {q^ \odot }M_{s,1}^{(t)}{({X^{(t - 1)}})^{\sum\limits_{v = t + 1,v \ne j}^n {{b_{vt}}{\e_v}}  + {\e_{n + t}}}}x_j^{{b_{jt}}},$$
where $	M_{s,1}^{(t)}  \in \mathbb{ZP}[x_1^{(n)}, \dots  ,x_{t - 1}^{(n)},x_{s + 1}, \dots ,{x_n}]$ and $\odot \in \frac{\mathbb{Z}}{2}.$
Then, according to formula (\ref{(qqq)}), we have
\begin{eqnarray*}
	f_{s, - {b_{tj}}}^{(t)} &=& \sum\limits_{v = 1}^{ - {b_{tj}}} {{q^{ - (v - 1){\lambda _{ts}}}}{{(x_t^{(n)})}^{ - {b_{tj}} - v}}f_{s,1}^{(t)}} {(x_t^{(n)})^{v - 1}}\\
	&	=& \sum\limits_{v = 1}^{ - {b_{tj}}} {{q^{ - (v - 1){\lambda _{ts}} +  \odot }}{{(x_t^{(n)})}^{ - {b_{tj}} - v}}M_{s,1}^{(t)}{{({X^{(t - 1)}})}^{\sum\limits_{i = t + 1,i \ne j}^n {{b_{it}}{\e_i}}  + {\e_{n + t}}}}x_j^{{b_{jt}}}} {(x_t^{(n)})^{v - 1}}\\
	&	=& L_tx_j^{{b_{jt}}},
\end{eqnarray*}
where $L_t  \in \mathbb{ZP}[x_1^{(n)}, \dots , x_{t }^{(n)}, x_{s+1}^{}, \dots ,{x_n}].$

Thus, we get
\begin{eqnarray*}
&&	{{q^{{M_j} - {\lambda _{js}} + {\lambda _{n + j\,s}} + \sum\limits_{v = j + 1}^n {{b_{vj}}{\lambda _{vs}}}  + \sum\limits_{r = t + 1}^{j - 1} {b_{rj}^{}{\lambda _{rs}}} }}\prod\limits_{i \in [1,t - 1]}^ \triangleleft  {{{(x_i^{(n)})}^{ - b_{ij}^{}}}} \cdot
f_{s, - {b_{tj}}}^{(t)}}\cdot \prod\limits_{i \in [t + 1,j - 1]}^ \triangleleft  {{{(x_i^{(n)})}^{ - b_{ij}^{}}}}\cdot  x_j^{ - 1}\\
	&=&  {{q^{{M_j} - {\lambda _{js}} + {\lambda _{n + j\,s}} + \sum\limits_{v= j + 1}^n {{b_{vj}}{\lambda _{vs}}}  + \sum\limits_{r = t + 1}^{j - 1} {b_{rj}^{}{\lambda _{rs}}} }}\prod\limits_{k \in [1,t - 1]}^ \triangleleft  {{{(x_k^{(n)})}^{ - b_{kj}^{}}}} \cdot L_tx_j^{{b_{jt}}}} \cdot \prod\limits_{k \in [t + 1,j - 1]}^ \triangleleft  {{{(x_k^{(n)})}^{ - b_{kj}^{}}}} \cdot x_j^{ - 1}\\
	&	=& {q^{\star}\prod\limits_{k \in [1,t - 1]}^ \triangleleft  {{{(x_k^{(n)})}^{ - b_{kj}^{}}}} \cdot L_t} \cdot \prod\limits_{k \in [t + 1,j - 1]}^ \triangleleft  {{{(x_k^{(n)})}^{ - b_{kj}^{}}}}
	\cdot x_j^{{b_{jt}} - 1},
\end{eqnarray*}
where $\star \in \frac{\mathbb{Z}}{2}.$
It follows that  $	M_{s,1}^{(j)} \in \mathbb{ZP}[x_1^{(n)}, \dots ,x_{j - 1}^{(n)},x_{s + 1}, \dots ,{x_n}].$

Note that
\begin{eqnarray*}
	{q^{ - {\lambda _{js}}}}{x_s}x_j^{(n)} &= &{q^{{M_j} - {\lambda _{js}}}}{x_s}\cdot \prod\limits_{k \in [1,j - 1]}^ \triangleleft  {{{(x_k^{(n)})}^{ - b_{kj}^{}}}} \cdot x_j^{ - 1}{({X^{(j - 1)}})^{\sum\limits_{t = j + 1}^n {{b_{tj}}{\e_t}} + {\e_{n + j}}}}
		+ {q^{ - {\lambda _{js}}}}{x_s}x_j^{ - 1}.
\end{eqnarray*}	

Hence, we obtain  $x_j^{(n)}{x_s} - {q^{ - {\lambda _{js}}}}{x_s}x_j^{(n)} = f_{s,1}^{(j)}.$ \end{proof}

By Lemma \ref{5201}, we have the following result. 	
\begin{corollary}\label{use}
 For any $j \in [1,n]$ and $ k< j,$ we have
$$f_{k,1}^{(j)}{\rm{ }} \in \mathbb{ZP}[x_1^{(n)}, \dots ,x_{j - 1}^{(n)},x_{k + 1}, \dots ,{x_n}].$$
\end{corollary}

\begin{remark}
Lemma \ref{5201} combining with Corollary \ref{use} is in the same spirit as in \cite{LS} which is described as Levendorski$\breve{{\i}}$-So$\breve{{\i}}$belman straightening law.
\end{remark}

\begin{lemma}\label{prove}
For any $j,k \in[1,n]$, $k<j$ and $l\in \mathbb{N}$,  we have
\begin{equation}\label{(jls')}
x_j^{(n)}x_k^l -{q^{ - l{\lambda _{jk}}}}x_k^lx_j^{(n)} =g_{k,l}^{(j)},
\end{equation}
where $g_{k,l}^{(j)} \in \mathbb{ZP}[x_1^{(n)}, \dots  ,x_{j - 1}^{(n)},x_{k}, \dots ,{x_n}].$
\end{lemma}	

\begin{proof}
According to Lemma \ref{5201}, we have
\begin{equation*}
	x_j^{(n)}x_k^l - {q^{ - l{\lambda _{jk}}}}x_k^lx_j^{(n)} = g_{k,l}^{(j)},
\end{equation*}	
where $g_{k,l}^{(j)} = \sum\limits_{t = 1}^l {{q^{ - (t - 1){\lambda _{jk}}}}x_k^{t-1}f_{k,1}^{(j)}} x_k^{l-t}.$

Thus, by Corollary \ref{use}, we obtain $g_{k,l}^{(j)} \in  \mathbb{ZP}[x_1^{(n)}, \dots  ,x_{j - 1}^{(n)},x_{k}, \dots ,{x_n}].$\end{proof}

\begin{definition}
 	A  monomial $x_1^{{a_1}} \cdots x_n^{{a_n}}{(x_1^{(n)})^{a_1^{(n)}}} \cdots {(x_n^{(n)})^{a_n^{(n)}}}$  is called  a projective standard  if  all exponents are non-negative integers and ${a_k}a_k^{(n)} = 0$ for all $ k \in [1, n]$.
\end{definition}

\begin{proposition}\label{span}
	For any  $j,k \in[1,n], a_k,{a_k^{(n)}} \in \mathbb{N}$ such that $a_k{a_k^{(n)}} =0$, the following items are all $\mathbb{ZP}$-linear combinations of  projective standard  monomials:
\begin{enumerate}
\item $x_j^{(n)}\prod\limits_{k \in [1,n]}^ \triangleleft  {x_k^{{a_k}}}\cdot \prod\limits_{k \in [1,n]}^ \triangleleft  {{{(x_k^{(n)})}^{a_k^{(n)}}}}$;
\item $\prod\limits_{k \in [1,n]}^ \triangleleft  {x_k^{{a_k}}}\cdot \prod\limits_{k \in [1,n]}^ \triangleleft  {{{(x_k^{(n)})}^{a_k^{(n)}}}}\cdot x_j^{(n)}$;
\item ${x_j}\prod\limits_{k \in [1,n]}^ \triangleleft  {x_k^{{a_k}}}\cdot \prod\limits_{k \in [1,n]}^ \triangleleft  {{{(x_k^{(n)})}^{a_k^{(n)}}}} $;
\item $ \prod\limits_{k \in [1,n]}^ \triangleleft  {x_k^{{a_k}}}\cdot \prod\limits_{k \in [1,n]}^ \triangleleft  {{{(x_k^{(n)})}^{a_k^{(n)}}}}\cdot {x_j}.$
\end{enumerate}
\end{proposition}

\begin{proof}
	Note that (2) and (3) can be  proved repeatedly by using the following exchange relation
\begin{equation*}
{x_k}x_k^{(n)}={q^{{ \bullet _k}}}\prod\limits_{j \in [k + 1,2n]}^ \triangleleft  {x_j^{{b_{jk}}}}  \cdot \prod\limits_{j \in [1,k - 1]}^ \triangleleft  {{{(x_j^{(k - 1)})}^{ - {b_{jk}}}}}  + 1,
\end{equation*}
where  $\bullet_k \in \frac{\mathbb{Z}}{2}.$

We  prove (1) by induction, and the proof of (4) follows from (1) and (3).

When $j=1$ and $a_1=0,$  the proof follows from the  equation
$$x_1^{(n)}\prod\limits_{k \in [1,n]}^ \triangleleft  {x_k^{{a_k}}} \cdot \prod\limits_{k \in [1,2n]}^ \triangleleft  {{{(x_k^{(n)})}^{a_k^{(n)}}}}
=q^{{\star}}\prod\limits_{k \in [1,n]}^ \triangleleft  {x_k^{{a_k}}} \cdot {{{(x_1^{(n)})}^{a_1^{(n)}+1}}}\cdot\prod\limits_{k \in [2,2n]}^ \triangleleft  {{{(x_k^{(n)})}^{a_k^{(n)}}}},$$
where $\star \in \frac{\mathbb{Z}}{2}.$

When $j=1$ and $a_1 \ne 0,$ the proof can be deduced from (3) and the following exchange relation:
\begin{equation*}
x_1^{(n)}{x_1}={q^{{\ast}}} \prod\limits_{j \in [2,2n]}^ \triangleleft  {x_j^{{b_{j1}}}}  + 1,
\end{equation*}
where $\ast \in \frac{\mathbb{Z}}{2}.$
	
Now suppose that $j\leq i-1,$  we have $x_{j}^{(n)}\prod\limits_{k \in [1,n]}^ \triangleleft  {x_k^{{a_k}}} \cdot \prod\limits_{k \in [1,n]}^ \triangleleft  {{{(x_k^{(n)})}^{a_k^{(n)}}}} $ is  $\mathbb{ZP}$-linear combination of  projective standard  monomials.
When $j=i$,  we have
\begin{eqnarray*}
&&	x_i^{(n)}\prod\limits_{k \in [1,n]}^ \triangleleft  {x_k^{{a_k}}} \cdot \prod\limits_{k \in [1,n]}^ \triangleleft  {{{(x_k^{(n)})}^{a_k^{(n)}}}}
	= x_i^{(n)}\prod\limits_{k \in [{1},{i-1}]}^ \triangleleft  {x_k^{{a_k}}} \cdot x_i^{{a_i}} \cdot \prod\limits_{k \in[i+1,n]}^ \triangleleft  {x_k^{{a_k}}} \cdot \prod\limits_{k \in [1,n]}^ \triangleleft  {{{(x_k^{(n)})}^{a_k^{(n)}}}} \\
	&\mathop  = \limits^{(\ref{(jls')})}& ({q^{ -a_{{1}} {\lambda _{i{1}}}}}x_{{1}}^{{a_{{1}}}}x_i^{(n)} + g_{{1},{a_{{1}}}}^{(i)})\prod\limits_{k \in [{2},{i-1}]}^ \triangleleft  {x_k^{{a_k}}} \cdot x_i^{{a_i}}\cdot \prod\limits_{k \in[i+1,n]}^ \triangleleft  {x_k^{{a_k}}} \cdot \prod\limits_{k \in [1,n]}^ \triangleleft  {{{(x_k^{(n)})}^{a_k^{(n)}}}} \\
&	= &{q^{ - a_{{1}}{\lambda _{i{1}}}}}x_{{1}}^{{a_{{1}}}}x_i^{(n)}\prod\limits_{k \in [{2},{i-1}]}^ \triangleleft  {x_k^{{a_k}}} \cdot x_i^{{a_i}} \cdot \prod\limits_{k \in[i+1,n]}^ \triangleleft  {x_k^{{a_k}}} \cdot \prod\limits_{k \in [1,n]}^ \triangleleft  {{{(x_k^{(n)})}^{a_k^{(n)}}}} \\
&&	+ g_{{1},{a_{{1}}}}^{(i)}\prod\limits_{k \in [{2},{i-1}]}^ \triangleleft  {x_k^{{a_k}}} \cdot x_i^{{a_i}} \cdot \prod\limits_{k \in[i+1,n]}^ \triangleleft  {x_k^{{a_k}}} \cdot \prod\limits_{k \in [1,n]}^ \triangleleft  {{{(x_k^{(n)})}^{a_k^{(n)}}}} \\
	&\mathop  = \limits^{(\ref{(jls')})} &{q^{ - \sum\limits_{t = 1}^{i-1} {{a_{{t}}\lambda _{i{t}}}} }}\prod\limits_{k \in [{1},{i-1}]}^ \triangleleft  {x_k^{{a_k}}} \cdot x_i^{(n)}x_i^{{a_i}} \cdot \prod\limits_{k \in[i+1,n]}^ \triangleleft  {x_k^{{a_k}}} \cdot \prod\limits_{k \in [1,n]}^ \triangleleft  {{{(x_k^{(n)})}^{a_k^{(n)}}}}  + {F},
\end{eqnarray*}		
where $${F} = \sum\limits_{t = 1}^{i-1} {{q^{ - \sum\limits_{p = 1}^{t - 1} {{a_{{p}}\lambda _{i{p}}}} }}\cdot \prod\limits_{k \in [{1},{{t - 1}}]}^ \triangleleft  {x_k^{{a_k}}} \cdot  g_{{t},{a_{{t}}}}^{(i)}\cdot \prod\limits_{k \in [{{t + 1}},{i-1}]}^ \triangleleft  {x_k^{{a_k}}} \cdot x_i^{{a_i}}\cdot \prod\limits_{k \in[i+1,n]}^ \triangleleft  {x_k^{{a_k}}} \cdot \prod\limits_{k \in[1,n]}^ \triangleleft  {{{(x_k^{(n)})}^{a_k^{(n)}}}} }.$$

By Lemma \ref{prove}, $ g_{t,a_{{t}}}^{(i)} \in \mathbb{ZP}[x_1^{(n)}, \dots  ,x_{i - 1}^{(n)},x_{t}, \dots ,{x_n}].$ Thus, by induction  and (3), we can deduce that $F$ is  $\mathbb{ZP}$-linear combination of  projective standard  monomials.

We claim that ${q^{ - \sum\limits_{t = 1}^{i-1} {{\lambda _{i{t}}}} }}\prod\limits_{k \in [{1},{i-1}]}^ \triangleleft  {x_k^{{a_k}}} \cdot x_i^{(n)}x_i^{{a_i}}\cdot \prod\limits_{k \in[i+1,n]}^ \triangleleft  {x_k^{{a_k}}} \cdot\prod\limits_{k \in [1,n]}^ \triangleleft  {{{(x_k^{(n)})}^{a_k^{(n)}}}}$ is  $\mathbb{ZP}$-linear combination of  projective standard  monomials.

(i) If $a_i=0$, the claim is obvious;

(ii) If $a_i\neq 0$, we  compute

\begin{eqnarray*}
	&& {q^{ - \sum\limits_{t = 1}^{i-1} {{\lambda _{i{t}}}} }}\prod\limits_{k \in [{1},{i-1}]}^ \triangleleft  {x_k^{{a_k}}} \cdot x_i^{(n)}x_i^{{a_i}} \cdot \prod\limits_{k \in[i+1,n]}^ \triangleleft  {x_k^{{a_k}}} \cdot \prod\limits_{k \in [1,n]}^ \triangleleft  {{{(x_k^{(n)})}^{a_k^{(n)}}}}\\
&	 =  &{q^{ - \sum\limits_{t = 1}^{i - 1} {{\lambda _{it}}} }}\prod\limits_{k \in [1,i - 1]}^ \triangleleft  {x_k^{{a_k}}} \cdot \big( {{q^{{ \circ _{i,1}}}}\prod\limits_{k \in [1,i - 1]}^ \triangleleft  {{{(x_k^{(i - 1)})}^{ - {b_{ki}}}}} \cdot \prod\limits_{k \in [i + 1,2n]}^ \triangleleft  {x_k^{{b_{ki}}}}  + 1} \big)\\
	&&\cdot x_i^{{a_i} - 1} \prod\limits_{k \in [i+1,n]}^ \triangleleft  {x_k^{{a_k}}} \cdot \prod\limits_{k \in [1,n]}^ \triangleleft  {{{(x_k^{(n)})}^{a_k^{(n)}}}}\\
	&=& {q^{{ \circ _{i,1}} - \sum\limits_{t = 1}^{i - 1} {{\lambda _{it}}} }}\prod\limits_{k \in [1,i - 1]}^ \triangleleft  {x_k^{{a_k}}} \cdot \prod\limits_{k \in [1,i - 1]}^ \triangleleft  {{{(x_k^{(i - 1)})}^{ - {b_{ki}}}}} \cdot \prod\limits_{k \in [i + 1,2n]}^ \triangleleft  {x_k^{{b_{ki}}}} \\
\end{eqnarray*}

\begin{eqnarray*}
	&&\cdot x_i^{{a_i} - 1}\prod\limits_{k \in [i+1,n]}^ \triangleleft  {x_k^{{a_k}}} \cdot \prod\limits_{k \in [1,n]}^ \triangleleft  {{{(x_k^{(n)})}^{a_k^{(n)}}}} \\
&&	+ {q^{ - \sum\limits_{t = 1}^{i - 1} {{\lambda _{it}}} }}\prod\limits_{k \in [1,i - 1]}^ \triangleleft  {x_k^{{a_k}}} \cdot x_i^{{a_i} - 1} \cdot \prod\limits_{k \in [i+1,n]}^ \triangleleft  {x_k^{{a_k}}} \cdot \prod\limits_{k \in [1,n]}^ \triangleleft  {{{(x_k^{(n)})}^{a_k^{(n)}}}}
\end{eqnarray*}
where ${ \circ _{i,1}} \in \frac{\mathbb{Z}}{2}.$
Thus, the claim follows by induction and (3). The proof is completed.\end{proof}

For any $\mathbf{a} = \left( {{a_1},{a_2}, \ldots, {a_n}} \right) \in {{\mathbb{Z}}^n},$ we denote
$${\mathbf{x}^\mathbf{a}} := x_1^{{a_1}}x_2^{{a_2}} \ldots x_n^{{a_n}}.$$

Let $\prec$ denote the lexicographic order on $\mathbb{Z}^{n}$, i.e., for any two vectors $\mathbf{a} = \left( {{a_1},{a_2}, \ldots, {a_n}} \right)$, $\mathbf{a}' = \left( {{a'_1},{a'_2}, \ldots, {a'_n}} \right)$ $\in\ZZ^n$ satisfy that $\mathbf{a}\prec\mathbf{a}'$ if and only if  there exists $k \in [1,n]$ such that $a_k<a'_k$ and $a_i=a'_i$ for all $i \in [1,k-1]$.
This order induces the lexicographic order on the Laurent monomials as
\begin{gather*}
	\mathbf{x}^{\mathbf{a}}\prec\mathbf{x}^{\mathbf{a}'}
	\quad \text{ if } \quad
	\mathbf{a}\prec\mathbf{a}'.
\end{gather*}

\begin{definition}	
Let $Y = g_{\mathbf{a}}{\mathbf{x}^\mathbf{a}}+\sum\limits_y^{} g_{\mathbf{a}_y}{{\mathbf{x}^{\mathbf{a}_y}}}$ for nonzero elements $g_{\mathbf{a}}, g_{\mathbf{a}_y}\in \mathbb{ZP}$. We call $g_{\mathbf{a}}{\mathbf{x}^\mathbf{a}}$  the first Laurent monomial of $Y$  if $\mathbf{a}_y\prec\mathbf{a}$ for any $y$ in some index set.	
\end{definition}

 \begin{theorem}\label{t3.6}
Let  ${\Sigma}=( {\widetilde{\mathbf{x}},\Lambda,{\widetilde B}})$ be an acyclic quantum seed. Then the projective standard monomials in
${x_1},x_1^{(n)}, \dots ,{x_n},x_n^{(n)}$ form a $\mathbb{ZP}$-basis of $\mathcal{A}( {\Sigma})$.
\end{theorem}

\begin{proof}
	According to  Theorem  \ref{thm1} and Proposition \ref{span}, we only need to prove that the projective
 standard monomials in ${x_1},x_1^{(n)}, \ldots ,{x_n},x_n^{(n)}$ are linearly independent over $\mathbb{ZP}$.
 	
We label projective standard monomials in
 ${x_1},x_1^{(n)}, \dots,{x_n},x_n^{(n)}$ by the points $\mathbf{a} = \left( {{a_1},{a_2}, \dots, {a_n}} \right) \in {{\mathbb{Z}}^n},$
  where ${\mathbf{x}^{\left\langle \mathbf{a} \right\rangle }} = x_1^{\left\langle {{a_1}} \right\rangle }x_2^{\left\langle {{a_2}} \right\rangle } \cdots x_n^{\left\langle {{a_n}} \right\rangle }$ and	
 	\[x_i^{\left\langle {{a_i}} \right\rangle } = \left\{ {\begin{array}{*{20}{c}}
 			{x_i^{{a_i}} ,   {\hspace{2.5cm}}if{\hspace{0.1cm}}{a_i} \ge 0};\\
 			{{{(x_i^{(n)})}^{ - {a_i}}},  {\hspace{1.5cm}}if{\hspace{0.1cm}}{a_i} < 0 }.
 \end{array}} \right.\]
	
Note that, projective standard monomials can be rewritten as $q^{{ \ast _{\mathbf{a}}}}{\mathbf{x}^{\left\langle \mathbf{a} \right\rangle }}$ for all $\mathbf{a} = \left( {{a_1},{a_2}, \dots, {a_n}} \right) \in {{\mathbb{Z}}^n}$, where
$\ast _{\mathbf{a}}\in  \frac{\mathbb{Z}}{2}$.
 For each ${a_i} < 0$, we have $$x_i^{\left\langle {{a_i}} \right\rangle }=(x_i^{(n)})^{ - {a_i}}=\big(x_{i}^{-1}({q^{{ \bullet _i}}} \prod\limits_{j \in [i + 1,2n]}^ \triangleleft  {x_j^{{b_{ji}}}}  \cdot \prod\limits_{j \in [1,i - 1]}^ \triangleleft  {{{(x_j^{(i - 1)})}^{ - {b_{ji}}}}}  + 1)\big)^{-{a_i}},$$
where
${\bullet _i}\in  \frac{\mathbb{Z}}{2}$.

 Thus, if $\mathbf{a}\prec\mathbf{a}'$, we can deduce that the first monomial in $\mathbf{x}^{\left\langle \mathbf{a} \right\rangle} $ precedes the first monomial in ${\mathbf{x}^{\left\langle \mathbf{a}' \right\rangle }}$. Hence, the proof is completed.
\end{proof}
\begin{remark}
\begin{enumerate}%[leftmargin=*]
\item If we set $q=1$ and $B$ skew-symmtric, one can obtain a $\mathbb{ZP}$-basis of the classical acyclic cluster algebra which is proved in \cite{BN}.
\item If $B$ is skew-symmtric, this basis should agree with an associated dual PBW basis in the language of \cite{KQ}. This is the reason why we call the basis in Theorem~\ref{t3.6} the dual PBW basis.
\end{enumerate}
\end{remark}

\section{An example}

Consider the acyclic quantum seed $({\widetilde{\mathbf{x}}},\Lambda,\widetilde{B})$ as follows:

\[\widetilde{B} = \left[ {\begin{array}{*{20}{c}}
		0&{ - 1}&{ - 1}\\
		1&0&{ - 2}\\
		1&2&0\\
		1&0&0\\
		0&1&0\\
		0&0&1
\end{array}}\,\, \right],\Lambda  = \left[ {\begin{array}{*{20}{c}}
		0&{ - 1}&{ - 1}&1&2&{ - 2}\\
		1&0&0&0&0&1\\
		1&0&0&0&1&0\\
		{ - 1}&0&0&0&{ - 1}&{ - 1}\\
		{ - 2}&0&{ - 1}&1&0&{ - 2}\\
		2&{ - 1}&0&1&2&0
\end{array}} \right].\]

We have
\begin{eqnarray*}
	{x'_1}&	=& {q^{ - \frac{1}{2}}}x_1^{ - 1}x_2^{}{x_3}{x_4} + x_1^{ - 1},\\
	{x'_2}&	=& {q^{ - 1}}x_2^{ - 1}x_3^2{x_5} + {q^{ - \frac{1}{2}}}x_1^{}x_2^{ - 1},\\
	{x'_3} &=&x_3^{ - 1}{x_6} + {q^{\frac{1}{2}}}x_1^{}x_2^2x_3^{ - 1}.
\end{eqnarray*}

By mutating the seed $({\widetilde{\mathbf{x}}},\Lambda,\widetilde{B})$ in the direction $1$, we have
\begin{eqnarray*}\label{(x1)}
	x_1^{(1)} &=& {x'_1} 	= {q^{ - \frac{1}{2}}}x_1^{ - 1}x_2^{}{x_3}{x_4} + x_1^{ - 1},
\end{eqnarray*}
\[{\widetilde{B}^{(1)}} = \left[ {\begin{array}{*{20}{c}}
		0&1&1\\
		{ - 1}&0&{ - 2}\\
		{ - 1}&2&0\\
		{ - 1}&0&0\\
		0&1&0\\
		0&0&1
\end{array}} \right], \,\,{\Lambda ^{(1)}} = \left[ {\begin{array}{*{20}{c}}
		0&1&1&{ - 1}&{ - 2}&2\\
		{ - 1}&0&0&0&0&1\\
		{ - 1}&0&0&0&1&0\\
		1&0&0&0&{ - 1}&{ - 1}\\
		2&0&{ - 1}&1&0&{ - 2}\\
		{ - 2}&{ - 1}&0&1&2&0
\end{array}} \right].\]

By mutating the seed $({\widetilde{\mathbf{x}}^{(1)}},{\Lambda ^{(1)}},{\widetilde{B}^{(1)}})$ in the direction $2$, we have
$$x_2^{(2)}= {q^{ - \frac{1}{2}}}x_1^{(1)}x_2^{ - 1}x_3^2{x_5} + x_2^{ - 1}={q^{ - 1}}x_1^{ - 1}x_3^3{x_4}{x_5} + {q^{ - \frac{1}{2}}}x_1^{ - 1}x_2^{ - 1}x_3^2{x_5} + x_2^{ - 1}.$$

Multiplying both sides of the above equation from left by $x_1$, we have
\begin{eqnarray*}
	x_1^{}x_2^{(2)} &=& {q^{ - 1}}x_3^3{x_4}{x_5} + {q^{ - \frac{1}{2}}}x_2^{ - 1}x_3^2{x_5} + x_1^{}x_2^{ - 1}\\
	&= &{q^{ - 1}}x_3^3{x_4}{x_5} + {q^{\frac{1}{2}}}({q^{ - 1}}x_2^{ - 1}x_3^2{x_5} + {q^{ - \frac{1}{2}}}x_1^{}x_2^{ - 1})\\
&	= &{q^{ - 1}}x_3^3{x_4}{x_5} + {q^{\frac{1}{2}}}{x'_2},
\end{eqnarray*}
\[{\widetilde{B}^{(2)}} = \left[ {\begin{array}{*{20}{c}}
		0&{ - 1}&1\\
		1&0&2\\
		{ - 1}&{ - 2}&0\\
		{ - 1}&0&0\\
		0&{ - 1}&0\\
		0&0&1
\end{array}} \right], \,\, {\Lambda ^{(2)}} = \left[ {\begin{array}{*{20}{c}}
		0&{ - 1}&1&{ - 1}&{ - 2}&2\\
		1&0&0&0&0&{ - 1}\\
		{ - 1}&0&0&0&1&0\\
		1&0&0&0&{ - 1}&{ - 1}\\
		2&0&{ - 1}&1&0&{ - 2}\\
		{ - 2}&1&0&1&2&0
\end{array}} \right].\]

By mutating the seed $({\widetilde{\mathbf{x}}^{(2)}},{\Lambda ^{(2)}},{\widetilde{B}^{(2)}})$ in the direction $3$, we have
\begin{eqnarray*}
			x_3^{(3)} &=& {q^{\frac{3}{2}}}x_1^{(2)}{(x_2^{(2)})^2}x_3^{ - 1}{x_6} + x_3^{ - 1}\\
			&=& {q^{\frac{3}{2}}}({q^{ - \frac{1}{2}}}x_1^{ - 1}x_2^{}{x_3}{x_4} + x_1^{ - 1}){({q^{ - \frac{1}{2}}}x_1^{(1)}x_2^{ - 1}x_3^2{x_5} + x_2^{ - 1})^2}x_3^{ - 1}{x_6} + x_3^{ - 1}\\
&=& {q^{ - 1}}x_1^{ - 1}{(x_1^{(1)})^2}x_2^{ - 1}x_3^4{x_4}x_5^2{x_6} + ({q^{\frac{1}{2}}} + {q^{\frac{3}{2}}})x_1^{ - 1}x_1^{(1)}x_2^{ - 1}x_3^2{x_4}{x_5}{x_6}\\
&&	+ {q^{}}x_1^{ - 1}x_2^{ - 1}{x_4}{x_6} + {q^{\frac{3}{2}}}x_1^{ - 1}{(x_1^{(1)})^2}x_2^{ - 2}x_3^3x_5^2{x_6}\\
	&&+ ({q^2} + {q^3})x_1^{ - 1}x_1^{(1)}x_2^{ - 2}x_3^{}{x_5}{x_6} + {q^{\frac{3}{2}}}x_1^{ - 1}x_2^{ - 2}x_3^{ - 1}{x_6} + x_3^{ - 1}.
\end{eqnarray*}

Multiplying both sides of the above equality from left by ${x_1}x_2^2$, we have
\begin{eqnarray*}
	{x_1}x_2^2x_3^{(3)} & =& {q^{ - 1}}{x_1}x_2^2x_1^{ - 1}{(x_1^{(1)})^2}x_2^{ - 1}x_3^4{x_4}x_5^2{x_6} + ({q^{\frac{1}{2}}} + {q^{\frac{3}{2}}}){x_1}x_2^2x_1^{ - 1}x_1^{(1)}x_2^{ - 1}x_3^2{x_4}{x_5}{x_6}\\
	&&+ {q^{}}{x_1}x_2^2x_1^{ - 1}x_2^{ - 1}{x_4}{x_6} + {q^{\frac{3}{2}}}{x_1}x_2^2x_1^{ - 1}{(x_1^{(1)})^2}x_2^{ - 2}x_3^3x_5^2{x_6}\\
	&&+ ({q^2} + {q^3}){x_1}x_2^2x_1^{ - 1}x_1^{(1)}x_2^{ - 2}x_3^{}{x_5}{x_6} + {q^{\frac{3}{2}}}{x_1}x_2^2x_1^{ - 1}x_2^{ - 2}x_3^{ - 1}{x_6} + {x_1}x_2^2x_3^{ - 1}\\
	&=& {q^{ - 7}}{(x_1^{(1)})^2}x_2^{}x_3^4{x_4}x_5^2{x_6} + ({q^{ - \frac{7}{2}}} + {q^{ - \frac{5}{2}}})x_1^{(1)}x_2^{}x_3^2{x_4}{x_5}{x_6}\\
	&&+ {q^{ - 1}}x_2^{}{x_4}{x_6} + {q^{ - \frac{9}{2}}}{(x_1^{(1)})^2}x_3^3x_5^2{x_6} + ({q^{ - 2}} + {q^{ - 1}})x_1^{(1)}x_3^{}{x_5}{x_6}\\
	&&+ {q^{ - \frac{1}{2}}}x_3^{ - 1}{x_6} + {x_1}x_2^2x_3^{ - 1}\\
&	= &{q^{ - 7}}{(x_1^{(1)})^2}x_2^{}x_3^4{x_4}x_5^2{x_6} + ({q^{ - \frac{7}{2}}} + {q^{ - \frac{5}{2}}})x_1^{(1)}x_2^{}x_3^2{x_4}{x_5}{x_6}\\
	&&+ {q^{ - 1}}x_2^{}{x_4}{x_6} + {q^{ - \frac{9}{2}}}{(x_1^{(1)})^2}x_3^3x_5^2{x_6} + ({q^{ - 2}} + {q^{ - 1}})x_1^{(1)}x_3^{}{x_5}{x_6}
	+ {q^{ - \frac{1}{2}}}{x'_3}.
\end{eqnarray*}

Thus
${x'_1},{x'_2}$ and ${x'_3} \in \mathbb{ZP}[{x_1},x_1^{(3)},{x_2},x_2^{(3)},{x_3},x_3^{(3)}].$
Note that
\begin{eqnarray*}\label{(e11)}
	& &x_1^{(3)}{x_1} - {x_1}x_1^{(3)}\nonumber\\
	& = &{q^{ - \frac{1}{2}}}x_1^{ - 1}x_2^{}{x_3}{x_4}{x_1} + x_1^{ - 1}{x_1} - {q^{ - \frac{1}{2}}}{x_1}x_1^{ - 1}x_2^{}{x_3}{x_4} - {x_1}x_1^{ - 1}\nonumber\\
	&	=& {q^{ - \frac{1}{2}}}(q - 1)x_2^{}{x_3}{x_4},\\
&&\\
	 & &x_2^{(3)}{x_2} - {x_2}x_2^{(3)} \nonumber\\
	 &=& {q^{ - \frac{1}{2}}}x_1^{(3)}x_2^{ - 1}x_3^2{x_5}{x_2} + x_2^{ - 1}{x_2}- {q^{ - \frac{1}{2}}}{x_2}x_1^{(3)}x_2^{ - 1}x_3^2{x_5} - {x_2}x_2^{ - 1}\nonumber\\
&	=& {q^{ - \frac{1}{2}}}(1 - {q^{ - 1}})x_1^{(3)}x_3^2{x_5},\\
&&\\
	& &x_3^{(3)}{x_3} - {x_3}x_3^{(3)} \nonumber\\
	&= &{q^{\frac{3}{2}}}x_1^{(3)}{(x_2^{(3)})^2}x_3^{ - 1}{x_6}{x_3} + x_3^{ - 1}{x_3}- {q^{\frac{3}{2}}}{x_3}x_1^{(3)}{(x_2^{(3)})^2}x_3^{ - 1}{x_6} - {x_3}x_3^{ - 1}\nonumber\\
	&=& {q^{\frac{1}{2}}}(q - 1)x_1^{(3)}{(x_2^{(3)})^2}{x_6},
\end{eqnarray*}
\begin{eqnarray*}\label{(e21)}
	& &x_2^{(3)}{x_1} - {q^{ - 1}}{x_1}x_2^{(3)} \nonumber\\
		&=& {q^{ - \frac{1}{2}}}x_1^{(3)}x_2^{ - 1}x_3^2{x_5}{x_1} + x_2^{ - 1}{x_1}- {q^{ - \frac{3}{2}}}{x_1}x_1^{(3)}x_2^{ - 1}x_3^2{x_5} - {q^{ - 1}}{x_1}x_2^{ - 1}\nonumber\\
	&=& {q^{ - \frac{3}{2}}}x_1^{(3)}{x_1}x_2^{ - 1}x_3^2{x_5} + {q^{ - 1}}{x_1}x_2^{ - 1} - {q^{ - \frac{3}{2}}}{x_1}x_1^{(3)}x_2^{ - 1}x_3^2{x_5} - {q^{ - 1}}{x_1}x_2^{ - 1}\nonumber\\
&	=& {q^{ - \frac{3}{2}}}(x_1^{(3)}{x_1} - {x_1}x_1^{(3)})x_2^{ - 1}x_3^2{x_5}\nonumber\\
	&=& {q^{ - 2}}(q - 1)x_3^3{x_4}{x_5},\\
&&\\
&&	x_3^{(3)}{x_2} - {x_2}x_3^{(3)}\nonumber\\
	&=& {q^{\frac{3}{2}}}x_1^{(3)}{(x_2^{(3)})^2}x_3^{ - 1}{x_6}{x_2} + x_3^{ - 1}{x_2} - {q^{\frac{3}{2}}}{x_2}x_1^{(3)}{(x_2^{(3)})^2}x_3^{ - 1}{x_6} - {x_2}x_3^{ - 1}\nonumber\\
&	=& {q^{\frac{1}{2}}}x_1^{(3)}{(x_2^{(3)})^2}{x_2}x_3^{ - 1}{x_6} - {q^{\frac{1}{2}}}x_1^{(3)}{x_2}{(x_2^{(3)})^2}x_3^{ - 1}{x_6}\nonumber\\
	&=& {q^{\frac{1}{2}}}x_1^{(3)}x_2^{(3)}\left( {{x_2}x_2^{(3)} + {q^{ - \frac{1}{2}}}(1 - {q^{ - 1}})x_1^{(3)}x_3^2{x_5}} \right)x_3^{ - 1}{x_6}\nonumber\\
&&	- {q^{\frac{1}{2}}}x_1^{(3)}\left( {x_2^{(3)}{x_2} - {q^{ - \frac{1}{2}}}(1 - {q^{ - 1}})x_1^{(3)}x_3^2{x_5}} \right)x_2^{(3)}x_3^{ - 1}{x_6}\\
&	=& (1 - {q^{ - 1}})x_1^{(3)}x_2^{(3)}x_1^{(3)}x_3^2{x_5}x_3^{ - 1}{x_6} + (1 - {q^{ - 1}}){(x_1^{(3)})^2}x_3^2{x_5}x_2^{(3)}x_3^{ - 1}{x_6}\\
	&= &({q^2} - 1){(x_1^{(3)})^2}x_2^{(3)}x_3^{}{x_5}{x_6},\\
&&\\
	&&x_3^{(3)}{x_1} - {q^{ - 1}}{x_1}x_3^{(3)}\\
	&= &{q^{\frac{3}{2}}}x_1^{(3)}{(x_2^{(3)})^2}x_3^{ - 1}{x_6}{x_1} + x_3^{ - 1}{x_1} - {q^{\frac{1}{2}}}{x_1}x_1^{(3)}{(x_2^{(3)})^2}x_3^{ - 1}{x_6} - {q^{ - 1}}{x_1}x_3^{ - 1}\\
	&=& {q^{\frac{5}{2}}}x_1^{(3)}{(x_2^{(3)})^2}{x_1}x_3^{ - 1}{x_6} + {q^{ - 1}}{x_1}x_3^{ - 1} - {q^{\frac{1}{2}}}{x_1}x_1^{(3)}{(x_2^{(3)})^2}x_3^{ - 1}{x_6} - {q^{ - 1}}{x_1}x_3^{ - 1}\\
	&=& {q^{\frac{5}{2}}}x_1^{(3)}x_2^{(3)}\left( {{q^{ - 1}}{x_1}x_2^{(3)} + {q^{ - 2}}(q - 1)x_3^3{x_4}{x_5}} \right)x_3^{ - 1}{x_6}\\
	&&- {q^{\frac{1}{2}}}\left( {x_1^{(3)}{x_1} - {q^{ - \frac{1}{2}}}(q - 1)x_2^{}{x_3}{x_4}} \right){(x_2^{(3)})^2}x_3^{ - 1}{x_6}\\
	&=& {q^{\frac{3}{2}}}x_1^{(3)}x_2^{(3)}{x_1}x_2^{(3)}x_3^{ - 1}{x_6} + {q^{\frac{1}{2}}}(q - 1)x_1^{(3)}x_2^{(3)}x_3^3{x_4}{x_5}x_3^{ - 1}{x_6}\\
	&&- {q^{\frac{1}{2}}}x_1^{(3)}{x_1}{(x_2^{(3)})^2}x_3^{ - 1}{x_6} + (q - 1)x_2^{}{x_3}{x_4}{(x_2^{(3)})^2}x_3^{ - 1}{x_6}\\
	&=& {q^{\frac{3}{2}}}x_1^{(3)}\left( {{q^{ - 1}}{x_1}x_2^{(3)} + {q^{ - 2}}(q - 1)x_3^3{x_4}{x_5}} \right)x_2^{(3)}x_3^{ - 1}{x_6}\\
	&&+ {q^{\frac{3}{2}}}(q - 1)x_1^{(3)}x_2^{(3)}x_3^2{x_4}{x_5}{x_6} - {q^{\frac{1}{2}}}x_1^{(3)}{x_1}{(x_2^{(3)})^2}x_3^{ - 1}{x_6}
	+ (q - 1)x_2^{}{(x_2^{(3)})^2}{x_4}{x_6}\\
&	=& {q^{ - \frac{1}{2}}}(q - 1)x_1^{(3)}x_3^3{x_4}{x_5}x_2^{(3)}x_3^{ - 1}{x_6} + {q^{\frac{3}{2}}}(q - 1)x_1^{(3)}x_2^{(3)}x_3^2{x_4}{x_5}{x_6}\\
	&&+ (q - 1)x_2^{}{(x_2^{(3)})^2}{x_4}{x_6}\\
	&=& (q - 1)x_2^{}\left( {{q^{ - \frac{1}{2}}}x_1^{(3)}x_2^{ - 1}x_3^2{x_5} + x_2^{ - 1}} \right)x_2^{(3)}{x_4}{x_6}
	+  {q^{\frac{1}{2}}}({q^2} - 1)x_1^{(3)}x_2^{(3)}x_3^2{x_4}{x_5}{x_6}\\
	&=& (q^{\frac{5}{2}}-q^{-\frac{1}{2}})x_1^{(3)}x_2^{(3)}x_3^2{x_4}{x_5}{x_6} + (q - 1)x_2^{(3)}{x_4}{x_6}
	.
\end{eqnarray*}

\section*{Acknowledgments}
Junyuan Huang was supported by Innovation Research for the Postgraduates of Guangzhou University (No. JCCX2024-053), Ming Ding was supported by NSF of China (No. 12371036) and Guangdong Basic and Applied Basic Research
Foundation (2023A1515011739) and Fan Xu was supported by NSF of China (No. 12031007).

\end{document}